\documentclass{article}
\usepackage{color}
\usepackage{amsmath}
\usepackage{amssymb}\usepackage{amsthm}
\usepackage{amsfonts}
\usepackage{amscd}
\begin{document}
\theoremstyle{plain}
\newtheorem*{ithm}{Theorem}
\newtheorem*{idefn}{Definition}
\newtheorem{thm}{Theorem}[section]
\newtheorem{lem}[thm]{Lemma}
\newtheorem{dlem}[thm]{Lemma/Definition}
\newtheorem{prop}[thm]{Proposition}
\newtheorem{set}[thm]{Setting}
\newtheorem{cor}[thm]{Corollary}
\newtheorem*{icor}{Corollary}
\theoremstyle{definition}
\newtheorem{assum}[thm]{Assumption}
\newtheorem{notation}[thm]{Notation}
\newtheorem{defn}[thm]{Definition}
\newtheorem{clm}[thm]{Claim}
\newtheorem{ex}[thm]{Example}
\theoremstyle{remark}
\newtheorem{rem}[thm]{Remark}
\newcommand{\unit}{\mathbb I}
\newcommand{\ali}[1]{{\mathfrak A}_{[ #1 ,\infty)}}
\newcommand{\alm}[1]{{\mathfrak A}_{(-\infty, #1 ]}}
\newcommand{\nn}[1]{\lV #1 \rV}
\newcommand{\br}{{\mathbb R}}
\newcommand{\dm}{{\rm dom}\mu}
\newcommand{\lb}{l_{\bb}(n,n_0,k_R,k_L,\lal,\bbD,\bbG,Y)}
\newcommand{\Ad}{\mathop{\mathrm{Ad}}\nolimits}
\newcommand{\Proj}{\mathop{\mathrm{Proj}}\nolimits}
\newcommand{\RRe}{\mathop{\mathrm{Re}}\nolimits}
\newcommand{\RIm}{\mathop{\mathrm{Im}}\nolimits}
\newcommand{\Wo}{\mathop{\mathrm{Wo}}\nolimits}
\newcommand{\Prim}{\mathop{\mathrm{Prim}_1}\nolimits}
\newcommand{\Primz}{\mathop{\mathrm{Prim}}\nolimits}
\newcommand{\ClassA}{\mathop{\mathrm{ClassA}}\nolimits}
\newcommand{\Class}{\mathop{\mathrm{Class}}\nolimits}
\newcommand{\diam}{\mathop{\mathrm{diam}}\nolimits}
\def\qed{{\unskip\nobreak\hfil\penalty50
\hskip2em\hbox{}\nobreak\hfil$\square$
\parfillskip=0pt \finalhyphendemerits=0\par}\medskip}
\def\proof{\trivlist \item[\hskip \labelsep{\bf Proof.\ }]}
\def\endproof{\null\hfill\qed\endtrivlist\noindent}
\def\proofof[#1]{\trivlist \item[\hskip \labelsep{\bf Proof of #1.\ }]}
\def\endproofof{\null\hfill\qed\endtrivlist\noindent}

\newcommand{\varphii}{\varphi}
\newcommand{\pgs}{\caP_{\sigma}}
\newcommand{\oo}{{\boldsymbol\varphii}}
\newcommand{\caA}{{\mathcal A}}
\newcommand{\caB}{{\mathcal B}}
\newcommand{\caC}{{\mathcal C}}
\newcommand{\caD}{{\mathcal D}}
\newcommand{\caE}{{\mathcal E}}
\newcommand{\caF}{{\mathcal F}}
\newcommand{\caG}{{\mathcal G}}
\newcommand{\caH}{{\mathcal H}}
\newcommand{\caI}{{\mathcal I}}
\newcommand{\caJ}{{\mathcal J}}
\newcommand{\caK}{{\mathcal K}}
\newcommand{\caL}{{\mathcal L}}
\newcommand{\caM}{{\mathcal M}}
\newcommand{\caN}{{\mathcal N}}
\newcommand{\caO}{{\mathcal O}}
\newcommand{\caP}{{\mathcal P}}
\newcommand{\caQ}{{\mathcal Q}}
\newcommand{\caR}{{\mathcal R}}
\newcommand{\caS}{{\mathcal S}}
\newcommand{\caT}{{\mathcal T}}
\newcommand{\caU}{{\mathcal U}}
\newcommand{\caV}{{\mathcal V}}
\newcommand{\caW}{{\mathcal W}}
\newcommand{\caX}{{\mathcal X}}
\newcommand{\caY}{{\mathcal Y}}
\newcommand{\caZ}{{\mathcal Z}}
\newcommand{\bbA}{{\mathbb A}}
\newcommand{\bbB}{{\mathbb B}}
\newcommand{\bbC}{{\mathbb C}}
\newcommand{\bbD}{{\mathbb D}}
\newcommand{\bbE}{{\mathbb E}}
\newcommand{\bbF}{{\mathbb F}}
\newcommand{\bbG}{{\mathbb G}}
\newcommand{\bbH}{{\mathbb H}}
\newcommand{\bbI}{{\mathbb I}}
\newcommand{\bbJ}{{\mathbb J}}
\newcommand{\bbK}{{\mathbb K}}
\newcommand{\bbL}{{\mathbb L}}
\newcommand{\bbM}{{\mathbb M}}
\newcommand{\bbN}{{\mathbb N}}
\newcommand{\bbO}{{\mathbb O}}
\newcommand{\bbP}{{\mathbb P}}
\newcommand{\bbQ}{{\mathbb Q}}
\newcommand{\bbR}{{\mathbb R}}
\newcommand{\bbS}{{\mathbb S}}
\newcommand{\bbT}{{\mathbb T}}
\newcommand{\bbU}{{\mathbb U}}
\newcommand{\bbV}{{\mathbb V}}
\newcommand{\bbW}{{\mathbb W}}
\newcommand{\bbX}{{\mathbb X}}
\newcommand{\bbY}{{\mathbb Y}}
\newcommand{\bbZ}{{\mathbb Z}}
\newcommand{\str}{^*}
\newcommand{\lv}{\left \vert}
\newcommand{\rv}{\right \vert}
\newcommand{\lV}{\left \Vert}
\newcommand{\rV}{\right \Vert}
\newcommand{\la}{\left \langle}
\newcommand{\ra}{\right \rangle}
\newcommand{\ltm}{\left \{}
\newcommand{\rtm}{\right \}}
\newcommand{\lcm}{\left [}
\newcommand{\rcm}{\right ]}
\newcommand{\ket}[1]{\lv #1 \ra}
\newcommand{\bra}[1]{\la #1 \rv}
\newcommand{\lmk}{\left (}
\newcommand{\rmk}{\right )}
\newcommand{\al}{{\mathcal A}}
\newcommand{\md}{M_d({\mathbb C})}
\newcommand{\ainn}{\mathop{\mathrm{AInn}}\nolimits}
\newcommand{\id}{\mathop{\mathrm{id}}\nolimits}
\newcommand{\Tr}{\mathop{\mathrm{Tr}}\nolimits}
\newcommand{\Ran}{\mathop{\mathrm{Ran}}\nolimits}
\newcommand{\Ker}{\mathop{\mathrm{Ker}}\nolimits}
\newcommand{\Aut}{\mathop{\mathrm{Aut}}\nolimits}
\newcommand{\spn}{\mathop{\mathrm{span}}\nolimits}
\newcommand{\Mat}{\mathop{\mathrm{M}}\nolimits}
\newcommand{\UT}{\mathop{\mathrm{UT}}\nolimits}
\newcommand{\DT}{\mathop{\mathrm{DT}}\nolimits}
\newcommand{\GL}{\mathop{\mathrm{GL}}\nolimits}
\newcommand{\spa}{\mathop{\mathrm{span}}\nolimits}
\newcommand{\supp}{\mathop{\mathrm{supp}}\nolimits}
\newcommand{\rank}{\mathop{\mathrm{rank}}\nolimits}
\newcommand{\idd}{\mathop{\mathrm{id}}\nolimits}
\newcommand{\ran}{\mathop{\mathrm{Ran}}\nolimits}
\newcommand{\dr}{ \mathop{\mathrm{d}_{{\mathbb R}^k}}\nolimits} 
\newcommand{\dc}{ \mathop{\mathrm{d}_{\cc}}\nolimits} \newcommand{\drr}{ \mathop{\mathrm{d}_{\rr}}\nolimits} 
\newcommand{\zin}{\mathbb{Z}}
\newcommand{\rr}{\mathbb{R}}
\newcommand{\cc}{\mathbb{C}}
\newcommand{\ww}{\mathbb{W}}
\newcommand{\nan}{\mathbb{N}}\newcommand{\bb}{\mathbb{B}}
\newcommand{\aaa}{\mathbb{A}}\newcommand{\ee}{\mathbb{E}}
\newcommand{\pp}{\mathbb{P}}
\newcommand{\wks}{\mathop{\mathrm{wk^*-}}\nolimits}
\newcommand{\mk}{{\Mat_k}}
\newcommand{\mnz}{\Mat_{n_0}}
\newcommand{\mn}{\Mat_{n}}
\newcommand{\dist}{\dc}
\newcommand{\braket}[2]{\left\langle#1,#2\right\rangle}
\newcommand{\ketbra}[2]{\left\vert #1\right \rangle \left\langle #2\right\vert}
\newcommand{\abs}[1]{\left\vert#1\right\vert}
\newtheorem{nota}{Notation}[section]
\def\qed{{\unskip\nobreak\hfil\penalty50
\hskip2em\hbox{}\nobreak\hfil$\square$
\parfillskip=0pt \finalhyphendemerits=0\par}\medskip}
\def\proof{\trivlist \item[\hskip \labelsep{\bf Proof.\ }]}
\def\endproof{\null\hfill\qed\endtrivlist\noindent}
\def\proofof[#1]{\trivlist \item[\hskip \labelsep{\bf Proof of #1.\ }]}
\def\endproofof{\null\hfill\qed\endtrivlist\noindent}
\newcommand{\ZZ}{\bbZ_2\times\bbZ_2}
\newcommand{\SSS}{\mathcal{S}}
\newcommand{\cs}{S}
\newcommand{\ct}{t}
\newcommand{\hS}{S}
\newcommand{\vv}{{\boldsymbol v}}
\newcommand{\ala}{a}
\newcommand{\bet}{b}
\newcommand{\gam}{c}
\newcommand{\alphas}{\alpha}
\newcommand{\alphai}{\alpha^{(\sigma_{1})}}
\newcommand{\alphan}{\alpha^{(\sigma_{2})}}
\newcommand{\betas}{\beta}
\newcommand{\betai}{\beta^{(\sigma_{1})}}
\newcommand{\betan}{\beta^{(\sigma_{2})}}
\newcommand{\alphass}{\alpha^{{(\sigma)}}}
\newcommand{\uu}{V}
\newcommand{\vp}{\varsigma}
\newcommand{\vpr}{R}
\newcommand{\tg}{\tau_{\Gamma}}
\newcommand{\sgg}{\Sigma_{\Gamma}^{(\sigma)}}
\newcommand{\nh}{1}
\newcommand{\rk}{2,a}
\newcommand{\nii}{1,a}
\newcommand{\nhh}{3,a}
\newcommand{\sjt}{2}
\newcommand{\sjtg}{2}
\newcommand{\bcg}{\caB(\caH_{\alpha})\otimes  C^{*}(\Sigma_{\Gamma}^{(\sigma)})}
\title{A classification of pure states on quantum spin chains satisfying the split property with on-site finite group symmetries}

\author{Yoshiko Ogata \thanks{ Graduate School of Mathematical Sciences
The University of Tokyo, Komaba, Tokyo, 153-8914, Japan
Supported in part by
the Grants-in-Aid for
Scientific Research, JSPS.}}
\maketitle

\begin{abstract}
We consider a set $SPG(\caA)$ of
pure split states on a quantum spin chain $\caA$
which are invariant under the on-site action $\tau$ of a finite group $G$
.
For each element $\omega$ in $SPG(\caA)$ we can associate a second cohomology 
 class $c_{\omega,R}$
of $G$. 
We consider a classification of $SPG(\caA)$ whose criterion is given as follows:
$\omega_{0}$ and $\omega_{1}$ in  $SPG(\caA)$
are equivalent if there are automorphisms $\Xi_{R}$, $\Xi_L$ on 
$\caA_{R}$, $\caA_{L}$ (right and left half infinite chains) preserving the symmetry $\tau$, such that
$\omega_{1}$ and $\omega_{0}\circ\lmk \Xi_{L}\otimes \Xi_{R}\rmk$
are quasi-equivalent.
It means that we can 
move $\omega_{0}$
close to $\omega_{1}$
without changing the entanglement nor breaking the symmetry.
We show that the second cohomology class $c_{\omega,R}$
is the complete invariant of this classification.

\end{abstract}
\section{Introduction}
It is well-known that the pure state space $P(\caA)$ of a quantum spin chain $\caA$
(UHF-algebra, see subsection \ref{setting})
is homogeneous under the action of the asymptotically inner automorphisms \cite{powers}, \cite{brat}, \cite{fkk}.
In fact, the homogeneity is proven for much larger class, i.e.,
for all the separable simple $C^{*}$-algebras \cite{kos}.
In this paper, we focus on the subset $SP(\caA)$ of $P(\caA)$ 
consisting of pure states satisfying the split property. (See Definition \ref{split}.)
One equivalent condition for a state $\omega\in P(\caA)$ to satisfy the split property is that
$\omega$ is quasi-equivalent to $\omega\vert_{\caA_{L}}\otimes \omega\vert_{\caA_{R}}$.
(See Remark \ref{splitrem}.)
Here, $\omega\vert_{\caA_{L}}$, $\omega\vert_{\caA_{R}}$ are restrictions of $\omega$
onto the left/right half-infinite chains. (See subsection \ref{setting}.)
A product state on $\caA=\caA_{L}\otimes \caA_{R}$ has no entanglement between 
$\caA_{L}$ and $\caA_{R}$ by definition.
In this sense, a  state with the split property has small entanglement between $\caA_{L}$
and $\caA_{R}$. Using the result  of  \cite{powers}, \cite{brat}, \cite{fkk},\cite{kos},
one can easily see that for any $\omega_{0},\omega_{1}\in SP(\caA)$,
there exist asymptotically inner automorphisms $\Xi_{L}$, $\Xi_{R}$ on $\caA_{L}$, $\caA_{R}$
such that $\omega_{1}\vert_{\caA_{L}}\sim_{q.e.}\omega_{0}\vert_{\caA_{L}}\circ \Xi_{L}$
and $\omega_{1}\vert_{\caA_{R}}\sim_{q.e.}\omega_{0}\vert_{\caA_{R}}\circ \Xi_{R}$.
(Here $\sim_{q.e.}$ means quasi-equivalence.)
From this and the split property of $\omega_{0}$, $\omega_{1}$, we see that $\omega_{1}$ and
$\omega_{0}\circ\lmk \Xi_{L}\otimes \Xi_{R}\rmk$ are quasi-equivalent.
The product of automorphisms $\Xi_{L}\otimes \Xi_{R}$ clearly does not
create/destroy any entanglement between $\caA_{L}$ and $\caA_{R}$.
Hence any $\omega_{0}\in SP(\caA)$
can get ''close to'' any $\omega_{1}\in SP(\caA)$ without changing the entanglement.
In this sense, we may regard $SP(\caA)$ to be ''homogeneous''.

What we would like to show in this paper is that the situation changes when symmetry comes into the game. This corresponds to the notion of symmetry protected topological phases in physics \cite{ogata}.
Let $SPG(\caA)$ be the set of all states in $ SP(\caA)$ which are invariant under the onsite action $\tau$
of a finite group $G$. (See Definition \ref{split}.)
We now require that  the automorphisms $\Xi_{L}$, $\Xi_{R}$ above to preserve the symmetry
i.e., $\Xi_{L}\circ \tau_{L}(g)= \tau_{L}(g)\circ \Xi_{L}$
and $\Xi_{R}\circ \tau_{R}(g)= \tau_{R}(g)\circ \Xi_{R}$ for all $g\in G$.
(See (\ref{tgg}) for the definition of $\tau_{L}$ and $\tau_{R}$.)
For any $\omega_{0},\omega_{1}\in SPG(\caA)$, can we always find such 
automorphisms giving $\omega_{1}\sim_{q.e.}\omega_{0}\circ\lmk \Xi_{L}\otimes \Xi_{R}\rmk$?
We show that the answer is no in general. 
The obstacle is given by the second cohomology class of the projective representation of $G$
associated to $\omega\in SPG(\caA)$.
We show that this second cohomology class is the complete invariant of this classification.

\subsection{Setting}\label{setting}
We consider the setting in this subsection throughout this paper.
We use the basic notation in section \ref{notasec} freely.
We start by summarizing standard setup of quantum spin chains on the infinite chain \cite{BR1,BR2}.
Throughout this paper, we fix some $2\le d\in\nan$.
We denote the algebra of $d\times d$ matrices by $\Mat_{d}$.

For each subset $\Gamma$ of $\bbZ$,
we denote the set of all finite subsets in $\Gamma$ by ${\mathfrak S}_{\Gamma}$.
We use the notation    $\Gamma_{R}=[0,\infty)\cap \bbZ$ and $\Gamma_{L}=(-\infty,-1]\cap \bbZ$.

For each $z\in\bbZ$,  let $\caA_{\{z\}}$ be an isomorphic copy of $\Mat_{d}$, and for any finite subset $\Lambda\subset\bbZ$, we set $\caA_{\Lambda} = \bigotimes_{z\in\Lambda}\caA_{\{z\}}$.
For finite $\Lambda$, the algebra $\caA_{\Lambda} $ can be regarded as the set of all bounded operators acting on
the Hilbert space $\bigotimes_{z\in\Lambda}{\bbC}^{d}$.
We use this identification freely.
If $\Lambda_1\subset\Lambda_2$, the algebra $\caA_{\Lambda_1}$ is naturally embedded in $\caA_{\Lambda_2}$ by tensoring its elements with the identity. 
For an infinite subset $\Gamma\subset \bbZ$,
$\caA_{\Gamma}$
is given as the inductive limit of the algebras $\caA_{\Lambda}$ with $\Lambda\in{\mathfrak S}_{\Gamma}$.
We call $\caA_{\Gamma}$ the quantum spin system on $\Gamma$.
In particular, we use notation
$\caA:=\caA_{\bbZ}$, $\caA_{R}:=\caA_{\Gamma_{R}}$ and
$\caA_{L}:=\caA_{\Gamma_{L}}$.
Occasionally, we call them quantum spin chain, right infinite chain, left infinite chain, respectively.
Note that each of $\caA_{\Lambda}$,  $\caA_{\Gamma}$ can be regarded naturally as a subalgebra of
$\caA$.
We also set $\caA_{\rm loc,\Gamma}=\bigcup_{\Lambda\in{\mathfrak S}_\Gamma}\caA_{\Lambda}
$, for any $\Gamma\subset \bbZ$.

We denote the standard basis of $\cc^{d}$ by $\{e_i\}_{i=1,\ldots,d}$, and 
denote the standard matrix unit of $\Mat_{d}$ by $\{E_{i,j}\mid i, j=1,\ldots,d\}$.
For each finite $\Lambda\subset \bbZ$, we denote
the tensor product $\bigotimes_{k\in\Lambda}E_{i_{k},j_{k}}$ of $E_{{i_{k},j_{k}}}$ along $k\in\Lambda$,
by $E_{I,J}^{(\Lambda)}$ with $I:=(i_{k})_{k\in\Lambda}$ and
 $J:=(j_{k})_{k\in\Lambda}$.
 We also use the notation 
\begin{align}\label{sldef}
\caS_{\Lambda}:=\left\{E_{I,J}^{(\Lambda)}\mid I,J\in \{1,\ldots, d\}^{\times \Lambda}
\right\}.
\end{align}
Furthermore, we set $e_{I}^{(\Lambda)}:=\bigotimes_{k\in\Lambda}e_{i_{k}}\in\bigotimes_{\Lambda}\bbC^{d}$
 for $I:=(i_{k})_{k\in\Lambda}$.

Throughout this paper we fix a finite group $G$ and its  unitary representation $U$ on $\bbC^{d}$
satisfying
\begin{align}\label{uffl}
U(g)\notin \bbC\unit_{{\bbC^{d}}},\quad \text{if} \quad g\neq e.
\end{align}
We denote the identity of $G$ by $e$.

Let $\Gamma\subset \bbZ$ be a non-empty subset.
For each $g\in G$, there exists a unique automorphism $\tau_{\Gamma}$ on $\caA_{\Gamma}$
such that 
\begin{align}\label{tgg}
\tau_{\Gamma}(g)\lmk a\rmk=\Ad\lmk\bigotimes_{I} U(g)\rmk\lmk a\rmk,\quad a\in\caA_{I},\quad g\in G,
\end{align}
for any finite subset $I$ of $\Gamma$.
We call the group homomorphism $\tau_{\Gamma}: G\to \Aut \caA_{\Gamma}$, 
the on-site action of $G$ 
on $\caA_{\Gamma}$ given by $U$.
In particular, when $\Gamma=\bbZ$, (resp. $\Gamma=\Gamma_{R}$, $\Gamma=\Gamma_{L}$),
we denote $\tau_{\Gamma}$ by $\tau$ (resp. $\tau_{R}$, $\tau_{L}$).
For $\Gamma\subset \bbZ$, we denote by $\caA_{\Gamma}^{G}$ the fixed point subalgebra of $\caA_{\Gamma}$
with respect to $\tau_{\Gamma}$.
For simplicity,  also use the notation $\caA_{L}^{G}:=\caA_{\Gamma_{L}}^{G}$
and $\caA_{R}^{G}:=\caA_{\Gamma_{R}}^{G}$.
\subsection{Projective representations of $G$}\label{pp}
A map $\sigma : G\times G\to \bbT$ is called a $2$-cocycle of $G$ if 
\begin{enumerate}
\item
$\sigma(g,h)\sigma(gh,k)=\sigma(h,k)\sigma(g,hk)$, for all $g,h,k\in G$,
\item 
$\sigma(g,e)=\sigma(e,g)=1$ for all $g\in G$.
\end{enumerate}
Define the product of two $2$-cocycles by their point-wise product.
The set of all $2$-cocycles of $G$ then becomes 
an abelian group.
The resulting group we denote by
$Z^{2}(G,\bbT)$.
The identity of $Z^{2}(G,\bbT)$ is given by $1_{Z^{2}(G,\bbT)}(g,h):=1$,
for $g,h\in G$.
For an arbitrary function $b: G\to \bbT$ such that $b(e)=1$,
\begin{align}\label{bound}
\sigma_{b}(g,h)=b(gh)^{-1} b(g)b(h), \quad g,h\in G
\end{align}
defines a $2$-cocycle.
The set of all 2-cocycles of this type
forms 
a normal subgroup
$B^{2}(G,\bbT)$ of $Z^{2}(G,\bbT)$.
The quotient group $H^{2}(G,\bbT):=Z^{2}(G,\bbT)/ B^{2}(G,\bbT)$ 
is called the second cohomology group of $G$.
For each $\sigma\in Z^{2}(G,\bbT)$, we denote by $[\sigma]_{H^{2}(G,\bbT)}$ the second cohomology class that $\sigma$ belongs to.

A projective unitary representation of $G$ 
is a triple $(\caH, \uu,\sigma)$ consisting of a Hilbert space $\caH$, 
a map $\uu: G \to \caU (\caH)$
 and a $2$-cocycle $\sigma$ of $G$
 such that $\uu(g)\uu(h) = \sigma(g, h)\uu(gh)$ for all $g,h\in G$. Note that we get $\uu(e)=\unit_{\caH}$ from the latter condition.
 We call $\sigma$, the $2$-cocycle of $G$ associated to $V$, and call
 $[\sigma]_{H^{2}(G,\bbT)}$ the second cohomology class of $G$
 associated to $V$. We occasionally say $(\caH,V)$ is a projective unitary representation with $2$-cocycle $\sigma$.
The character of  a finite dimensional projective unitary representation $(\caH,V,\sigma)$
is given by $\chi_{V}(g)=\Tr_{\caH}V(g)$, for $g\in G$.
 
 We say a projective unitary representation  
$(\caH, \uu,\sigma)$ of $G$ is irreducible if $\caH$ and $0$ are the only $\uu$-invariant subspaces of $\caH$.
As $G$ is a finite group, for any irreducible projective unitary  representation $(\caH, \uu,\sigma)$ of $G$, the Hilbert space
$\caH$ is finite dimensional.
Projective unitary representations $(\caH_{1}, \uu_{1},\sigma_{1})$ and
$(\caH_{2}, \uu_{2},\sigma_{2})$ are said to be unitarily equivalent if there is a unitary
$W:\caH_{1}\to\caH_{2}$ such that $W\uu_{1}(g)W^{*}= \uu_{2}(g)$, with $g\in G$.
Clearly if $(\caH_{1}, \uu_{1},\sigma_{1})$ and
$(\caH_{2}, \uu_{2},\sigma_{2})$ are unitarily equivalent, the $2$-cocycles $\sigma_{1}$ and $\sigma_{2}$ coincides.
Schur's Lemma holds: let $(\caH_{1}, \uu_{1},\sigma_{1})$ and
$(\caH_{2}, \uu_{2},\sigma_{2})$ be irreducible projective unitary representations of $G$, and
$W:\caH_{1}\to\caH_{2}$ be a linear map such that
$W \uu_{1}(g)=\uu_{2}(g) W$ for all $g\in G$.
Then either $V=0$ or 
$(\caH_{1}, \uu_{1},\sigma_{1})$ and
$(\caH_{2}, \uu_{2},\sigma_{2})$ are unitarily equivalent.
The proof is the same as that of the genuine representations (see \cite{simon} Theorem II.4.2
for example.)

For $\sigma\in Z^{2}(G,\bbT)$, we denote by $\pgs$,
the set of all unitarily equivalence classes of irreducible projective representations
with $2$-cocycle $\sigma$. 
Note that $\caP_{1_{Z^{2}(G,\bbT)}}$ is equal to $\hat G$,
the dual of $G$.

For each $\alpha\in \pgs$,
we fix a representative $(\caH_{\alphas}, \uu_{\alphas},\sigma)$.
We denote the dimension of $\caH_{\alphas}$ 
(which is finite) by $n_{\alphas}$ and
fix an orthonormal basis 
$\{\psi_k^{(\alphas)}\}_{k=1}^{n_{\alphas} }$ of $\caH_{\alphas}$.
We introduce the matrix unit $\{f_{k,j}^{(\alpha)}\mid k,j=1,\ldots, n_{\alpha}\}$
 of $\caB(\caH_{\alpha})$ given by
 \begin{align}
 f_{k,j}^{(\alpha)}\xi=\braket{\psi_j^{(\alphas)}}{\xi}\psi_k^{(\alphas)},\quad
 \xi\in \caH_{\alpha}.\quad k,j=1,\ldots, n_{\alpha}.
 \end{align}

We will use the following vector later, in section \ref{homogeneity}
\begin{align}\label{oa}
\Omega_{\alpha}:=
\frac 1{\sqrt{n_{\alpha}}}
\sum_{k=1}^{n_{\alpha}}
\psi_k^{(\alphas)}\otimes \psi_k^{(\alphas)}\in \caH_{\alphas}\otimes \caH_{\alphas}.
\end{align}
For each $\alphas\in \pgs$
and $k,j=1,\ldots, n_{\alphas}$, define a function $\lmk \uu_{\alphas}\rmk_{k,j}$
on $G$ by
\begin{align}
\lmk \uu_{\alphas}\rmk_{k,j}(g)
:=\braket{\psi_k^{(\alphas)}}{ \uu_{\alphas}(g){\psi_j^{(\alphas)}}},\quad
g\in G.
\end{align}
As in Theorem III.1.1 of \cite{simon}, from Schur's Lemma, 
we obtain the orthogonality relation:
\begin{align}\label{orthog}
\frac{1}{|G|} \sum_{g\in G}
\lmk \uu_{\alphas}\rmk_{k,j}(g)
\overline{\lmk \uu_{\betas}\rmk_{t,s}(g)}
=\frac{\delta_{\alphas,\betas}\delta_{j,s}\delta_{k,t}}{ n_{\alphas}},
\end{align}
for all $\alphas,\betas\in \pgs$
and $k,j,t,s=1,\ldots, n_{\alphas}$.
Here $|G|$ denotes the number of elements in $G$.
In particular, $\pgs$ is a finite set.
We freely identify $\alpha$ and $V_{\alpha}$.
For example, $\alpha\otimes \beta'$, $\alpha\otimes V$ should be understood as 
$V_{\alpha}\otimes V_{\beta'}$, 
$V_{\alpha}\otimes V$
for $\alpha\in\pgs$, $\beta'\in \caP_{\sigma'}$, and a projective unitary representation $V$.
We repeatedly use the following fact.
\begin{lem}\label{pd}
For any projective unitary  representation $(\caH, \uu,\sigma)$, there are
Hilbert spaces $\caK_{\alpha}$ labeled by $\alpha\in\pgs$
and a unitary $W: \caH\to \bigoplus_{\alpha\in \pgs}\caH_{\alpha}\otimes \caK_{\alpha}$
such that
\begin{align}\label{osos}
W\uu(g)W^{*}=\bigoplus_{\alpha\in \pgs}
\uu_{\alpha}(g)\otimes \unit_{\caK_{\alpha}},\quad g\in G.
\end{align}
Furthermore, the commutant $\uu(G)':=\{ X\in \caB(\caH)\mid [X, \uu(g)]=0\}$
of $\uu(G)$ is of the form
\begin{align}\label{com}
\uu(G)'=W^{*}\lmk
\bigoplus_{\alpha\in \pgs}
\unit_{\caH_{\alpha}}\otimes \caB\lmk {\caK_{\alpha}}\rmk\rmk W
\end{align}
\end{lem}
\begin{proof}
For any $\uu$-invariant subspace of $\caH$, its orthogonal complement is $\uu$-invariant as well.
Therefore, from
Zorn's Lemma, we may decompose $(\caH, \uu,\sigma)$
as an orthogonal sum of irreducible projective unitary representations with $2$-cocycle $\sigma$.
This proves (\ref{osos}).
The second statement (\ref{com}) follows from the orthogonality relation (\ref{orthog}).
\end{proof}

\begin{notation}\label{ddcom}
When (\ref{osos}) holds, we say that 
$V$ (or $(\caH, \uu,\sigma)$) has an irreducible decomposition given by
Hilbert spaces $\{\caK_{\gamma}\mid \gamma \in  \pgs\}$.
We say $V$ (or $(\caH, \uu,\sigma)$) contains all elements of $\pgs$ if $\caK_{\alpha}\neq \{0\}$ for 
all $\alpha\in \pgs$.
We say $V$ (or $(\caH, \uu,\sigma)$) contains all elements of $\pgs$ with infinite multiplicity
if $\dim \caK_{\alpha}=\infty$ for 
all $\alpha\in \pgs$.
We hence force  omit $W$ in  (\ref{osos}) and identify $\caH$ and $\bigoplus_{\alpha\in \pgs}\caH_{\alpha}\otimes \caK_{\alpha}$ freely.
The Hilbert space $\caH_{\alpha}\otimes \caK_{\alpha}$ can be naturally regarded as a closed subspace of $\caH$. We use this identification freely and call $\caH_{\alpha}\otimes \caK_{\alpha}$ the $\alpha$-component of $V$ (or $(\caH, \uu,\sigma)$).
\end{notation}

\begin{notation}\label{bfvee}
Let $(\caH, \uu,\sigma)$ be
a projective unitary representation. Let $b:G\to \bbT$ be a map such that $b(e)=1$.
Setting $\sigma_{b}$ as in (\ref{bound}),
we obtain $\sigma\sigma_{b}\in  Z^{2}(G,\bbT)$.
We also set $\lmk b\cdot \uu\rmk (g):=b(g)\uu(g)$, for $g\in G$. 
Then
$(\caH, b\cdot \uu,\sigma\sigma_b)$ is a projective representation. 
\end{notation}

\subsection{The split property and projective representations}
Next let us introduce the split property.
\begin{defn}\label{split}
Let $\omega$ be a pure state on $\caA$. Let $\omega_R$ be the restriction of
$\omega$ to $\caA_R$, and $(\caH_{\omega_R},\pi_{\omega_R},\Omega_{\omega_R})$ be the GNS triple of $\omega_R$.
We say $\omega$ satisfies the split property with respect to $\caA_L$ and $\caA_R$,
if the von Neumann algebra $\pi_{\omega_R}(\caA_{R})''$ is a type I factor.
We denote by $SP(\caA)$ the set of all pure states on $\caA$
which satisfy the split property with respect to $\caA_L$ and $\caA_R$.
We also denote by $SPG(\caA)$, the set of all states $\omega$ in $SP(\caA)$, which are 
$\tau$-invariant.
\end{defn}
Recall that a type I factor is $*$-isomorphic to $B(\caK)$, the set of all bounded operators 
on a  Hilbert space $\caK$. 
See \cite{takesaki}.
\begin{rem}\label{splitrem}
Let $\omega$ be a pure state on $\caA$. 
Let  $\omega_L$ be
the restriction of $\omega$ to $\caA_L$.
Then $\omega$ satisfies the split property if and only if 
$\omega_{L}\otimes \omega_{R}$ is quasi-equivalent to $\omega$. ( See \cite{Matsui2}. In Proposition 2.2 of \cite{Matsui2}, it is assumed that the state is translationally invariant because of the first equivalent condition (i). However, the proof for the equivalence between
(ii) and (iii) does not require translation invariance.)
Therefore, by the symmetric argument, if $(\caH_{\omega_L},\pi_{\omega_L},\Omega_{\omega_L})$ is the GNS triple of $\omega_L$,
the the split property of $\omega$ implies that $\pi_{\omega_L}(\caA_{L})''$ is also a type I factor.
\end{rem}
For each $\omega\in SPG(\caA)$, we may associate a second cohomology class of $G$.
\begin{prop}\label{unieq}
Let $\omega\in SPG(\caA)$ and ${\vp}=L,R$. 
Then
there exists an irreducible $*$-representation  $\rho_{\omega,{\vp}}$ 
of $\caA_{{\vp}}$ on a Hilbert space
 $\caL_{\omega,{\vp}}$
that is quasi-equivalent to the GNS representation 
of $\omega\vert_{\caA_{{\vp}}}$.
For each of such irreducible $*$-representation $(\caL_{\omega,{\vp}}, \rho_{\omega,{\vp}})$,
there is a projective unitary representation $u_{\omega,{\vp}}$ of $G$ 
on $\caL_{\omega,{\vp}}$ such that
\begin{align}\label{uintro}
\rho_{\omega,{\vp}}\circ \tau_{\vp}(g)
=\Ad\lmk u_{\omega,{\vp}}(g)\rmk\circ \rho_{\omega,{\vp}},
\end{align}
for all  $g\in G$.
Furthermore,
if  another triple $(\tilde \caL_{\omega,{\vp}}, \tilde\rho_{\omega,{\vp}}, \tilde u_{\omega,{\vp}})$ satisfies the same conditions as $(\caL_{\omega,{\vp}}, 
\rho_{\omega,{\vp}}, u_{\omega,{\vp}})$ above, then
there is a unitary $W:\caL_{\omega,\vp}\to \tilde \caL_{\omega,\vp}$ and $c: G\to  \bbT$
such that
\begin{align}\label{Wp}
\lmk \Ad W\rmk\circ  \rho_{\omega,\vp} =
\tilde \rho_{\omega,\vp},
  \end{align}
 \begin{align}\label{cg}
 c(g) \cdot \lmk \Ad W\rmk\lmk u_{\omega,\vp}(g)\rmk=\tilde u_{\omega,\vp}(g),\quad g\in G.
 \end{align}
In particular, 
for $2$-cocycle $\sigma_{\omega,\vp}$, $\tilde \sigma_{\omega,\vp}$
associated to $u_{\omega,\vp}$, $ \tilde u_{\omega,\vp}$ respectively,
we have
$[\sigma_{\omega,\vp}]_{H^{2}(G,\bbT)}=[\tilde \sigma_{\omega,\vp}]_{H^{2}(G,\bbT)}$.
\end{prop}
\begin{proof}
Let $(\caH_{\omega_{\vp}},\pi_{\omega_{\vp}},\Omega_{\omega_{\vp}})$ be the GNS triple of $\omega\vert_{\caA_{{\vp}}}$.
The existence of irreducible $*$-representation $(\caL_{\omega,{\vp}}, \rho_{\omega,{\vp}})$
quasi-equivalent to $\pi_{\omega_{\vp}}$ 
follows from the definition of the split property. 

To see the existence of $ u_{\omega,\vp} $ satisfying (\ref{uintro})
for such $(\caL_{\omega,{\vp}}, \rho_{\omega,{\vp}})$,
let $\iota_{\omega,{\vp}}: \pi_{\omega_{\vp}}\lmk \caA_{\vp}\rmk''
\to \caB(\caL_{\omega,{\vp}})$ be the $*$-isomorphism
such that
$\rho_{\omega,{\vp}}=\iota_{\omega,{\vp}}\circ  \pi_{\omega_{\vp}}$. 
By the $\tau_{\vp}$-invariance of $\omega\vert_{\caA_{{\vp}}}$,
the action $\tau_{\vp}$ of $G$ can be extended to an action
$\hat\tau_{\vp}$ on $\pi_{\omega_{\vp}}\lmk \caA_{\vp}\rmk''$,
so that $\hat\tau_{\vp}(g)\circ\pi_{\omega_{\vp}}=\pi_{\omega_{\vp}}\circ \tau_{\vp}(g)$,
for $g\in G$.
By the Wigner Theorem, the $*$-automorphism $\iota_{\omega,{\vp}}\circ \hat\tau_{\vp}(g)\circ
\iota_{\omega,{\vp}}^{-1}$ on $\caB(\caL_{\omega,{\vp}})$
is given by a unitary $ u_{\omega,\vp} (g)$
so that
\begin{align}
\iota_{\omega,{\vp}}\circ \hat\tau_{\vp}(g)\circ
\iota_{\omega,{\vp}}^{-1}
=\Ad\lmk u_{\omega,\vp} (g)
\rmk,\quad g\in G.
\end{align}
As $\hat \tau_{\vp}$ is an action of $G$, $u_{\omega,\vp}$ is a projective unitary representation.
We obtain (\ref{uintro}) by
\begin{align}
\rho_{\omega,{\vp}}\circ \tau_{\vp}(g)
=\iota_{\omega,{\vp}}\circ  \pi_{\omega_{\vp}}\circ \tau_{\vp}(g)
=\iota_{\omega,{\vp}}\circ  \hat\tau_{\vp}(g)\circ\pi_{\omega_{\vp}}
=\iota_{\omega,{\vp}}\circ  \hat\tau_{\vp}(g)\circ\iota_{\omega,{\vp}}^{-1}\circ
\iota_{\omega,{\vp}}\circ\pi_{\omega_{\vp}}
=\Ad\lmk u_{\omega,{\vp}}(g)\rmk\circ \rho_{\omega,{\vp}}.
\end{align}

Suppose that $(\tilde \caL_{\omega,{\vp}}, \tilde\rho_{\omega,{\vp}}, \tilde u_{\omega,{\vp}})$ satisfies the same conditions as $(\caL_{\omega,{\vp}}, 
\rho_{\omega,{\vp}}, u_{\omega,{\vp}})$.
Then by the Wigner Theorem, there exists a unitary $W:\caL_{\omega,\vp}\to \tilde \caL_{\omega,\vp}$
satisfying (\ref{Wp}).
Note that
\begin{align}
\Ad\lmk \tilde u_{\omega,{\vp}}(g)\rmk\circ \tilde \rho_{\omega,{\vp}}
=
\tilde \rho_{\omega,{\vp}}\circ \tau_{\vp}(g)
=\Ad W\circ  \rho_{\omega,\vp} 
\circ \tau_{\vp}(g)
=\Ad W\circ 
\Ad\lmk u_{\omega,{\vp}}(g)\rmk
\circ \Ad W^*\circ \tilde \rho_{\omega,{\vp}}.
\end{align}
This implies that $\tilde u_{\omega,{\vp}}(g)^*\Ad W\lmk u_{\omega,{\vp}}(g) \rmk$
belongs to $\bbT\unit_{\tilde \caL_{\omega,\vp}}$ proving (\ref{cg}).
\end{proof}
\begin{defn}\label{pao}Let $\omega\in SPG(\caA)$ and
$(\caL_{\omega,{\vp}}, \rho_{\omega,{\vp}}, u_{\omega,{\vp}})$ 
be a triple satisfying the conditions in Proposition \ref{unieq}.
Let $\sigma_{\omega,{\vp}}$ be the $2$-cocycle associated to
$u_{\omega,{\vp}}$.
We call 
$(\caL_{\omega,{\vp}}, \rho_{\omega,{\vp}}, u_{\omega,{\vp}},\sigma_{\omega,{\vp}})$  
a quadruple associated to $(\omega\vert_{\caA_{{\vp}}},\tau_{\vp})$.
Furthermore, we denote the second cohomology class $[\sigma_{\omega,{\vp}}]_{H^{2}(G,\bbT)}$
by $c_{\omega,\zeta}$, and call it the second cohomology class of $G$ associated to
$(\omega\vert_{\caA_{{\vp}}},\tau_{\vp})$.
\end{defn}
\begin{rem}\label{rem17}
For a quadruple $(\caL_{\omega,{\vp}}, \rho_{\omega,{\vp}}, u_{\omega,{\vp}},\sigma_{\omega,{\vp}})$
 associated to $(\omega\vert_{\caA_{{\vp}}},\tau_{\vp})$
 and any map $b:G\to\bbT$,
 $(\caL_{\omega,{\vp}}, \rho_{\omega,{\vp}}, b\cdot u_{\omega,{\vp}},\sigma_{b}\sigma_{\omega,{\vp}})$
 is also a quadruple associated to  
 $(\omega\vert_{\caA_{{\vp}}},\tau_{\vp})$.
 See (\ref{bound}) and Notation \ref{bfvee}.
\end{rem}

\subsection{Main Theorem}
Let us introduce $\ainn^{G} (A_{\vp})$.
\begin{defn}
Let $\vp=L,R$.
An automorphism $\Xi_{\vp}$ of $\caA_{\vp}$ is asymptotically inner in $\caA_{\vp}^{G}$
if there is a norm continuous path $w_{\vp}:[0,\infty)\to \caU\lmk\caA_{\vp}^{G}\rmk$
with $w_{\vp}(0)=\unit_{\caA_{\vp}}$
 that 
\begin{align}
\Xi_{\vp}(a)=\lim_{t\to\infty}\Ad\lmk w_{\vp}(t)\rmk(a),\quad
a\in \caA_{\vp}.
\end{align}
We denote by $\ainn^{G} (A_{\vp})$ the set of all automorphisms which are
 asymptotically inner in $\caA_{\vp}^{G}$.
\end{defn}
In this paper, we consider the classification problem of $SPG(\caA)$
with respect to the following equivalence relation.
\begin{defn}
For $\omega_{0}$, $\omega_{1}\in SPG(\caA)$, we
write $\omega_{0}\sim_{{\rm split},\tau}\omega_{1}$
if there exist automorphisms $\Xi_{L}\in \ainn^{G}(\caA_{L})$ 
and $\Xi_{R}\in \ainn^{G}(\caA_{R})$ such that
$\omega_{1}$ and $\omega_{0}\circ\lmk\Xi_{L}\otimes \Xi_{R}\rmk$
are quasi-equivalent.
\end{defn}

Now we are ready to state our main theorem.
\begin{thm}\label{main}
For  $\omega_{0}$, $\omega_{1}\in SPG(\caA)$,
$\omega_{0}\sim_{{\rm split},\tau}\omega_{1}$
if and only if
$c_{\omega_{1},R}=c_{\omega_{0},R}$.

\end{thm}
The ''only if'' part of the Theorem \ref{main} is easy to prove.
 In order to prove ''if'' part of the Theorem, we note that
if $c_{\omega_{1},R}=c_{\omega_{0},R}$ holds,
 $\omega_{0}$ and $\omega_{1}$
give covariant representations of a twisted $C^{*}$-dynamical systems 
 $\Sigma_{\Gamma_{R}}^{(\sigma_{R})}$, $\Sigma_{\Gamma_{L}}^{(\sigma_{L})}$
 (see section \ref{crossedproduct}), where
 $\sigma_{R}$, $\sigma_{L}$ are
  $2$-cocycles of $G$  such that
 $[\sigma_{R}]_{H^{2}(G,\bbT)}=c_{\omega_{1},R}=c_{\omega_{0},R}$
 and $[\sigma_{L}]_{H^{2}(G,\bbT)}=c_{\omega_{1},L}=c_{\omega_{0},L}$.
 (See Remark \ref{rem17} and Lemma \ref{clmcr}.)
 One of the basic idea is to encode the information of these 
 $2$-cocycles $\sigma_{R}$, $\sigma_{L}$
 into $C^{*}$-algebras we consider.
 Namely, instead of considering $\caA_{R}$, $\caA_{L}$,we consider 
 the  the twisted crossed products $C^{*}(\Sigma_{\Gamma_{R}}^{(\sigma_{R})})$,
 $C^{*}(\Sigma_{\Gamma_{L}}^{(\sigma_{L})})$.
 We recall the twisted crossed product $C^{*}(\Sigma_{\Gamma}^{(\sigma)})$ of $\Sigma_{\Gamma}^{(\sigma)}$ in section \ref{crossedproduct}.
 In section  \ref{projsec}, we show that
for any $\omega\in SPG(\caA)$, and $\vp=L,R$, $u_{\omega,{\vp}}$
contains all elements of $\caP_{\sigma_{\omega,{\vp}}}$.
Therefore, for any fixed $\alpha_{\vp}\in \caP_{\sigma_{{\vp}}}$, both of $u_{\omega_{0},{\vp}}$ and $u_{\omega_{1},{\vp}}$
 contains $\alpha_{\vp}$.
 This fact allows us to regard 
 the problem as the homogeneity problem of 
$ \caB(\caH_{\alpha_{\vp}})\otimes C^{*}(\Sigma_{\Gamma_{\vp}}^{(\sigma_{\vp})})$,
with symmetry
(section \ref{homogeneity}).
The proof of the homogeneity relies on the 
machinery developed in 
\cite{powers}, \cite{brat}, \cite{fkk},\cite{kos}.
However, for our problem, we would like to take the path of
unitaries in the fixed point algebras $\caA_{R}^{G}$, $\caA_{L}^{G}$.
This requires some additional argument using the irreducible decompositions of 
$u_{\omega_{0},{\vp}}$, $u_{\omega_{1},{\vp}}$.
This is given in section \ref{homogeneity}.
 
\section{Irreducible components in $u_{\omega,\vp}$}\label{projsec}

In this section we show that 
$u_{\omega,\vp}$ contains all elements in 
$\caP_{\sigma_{\omega,\vp}}$ with infinite multiplicity.

 As $G$ is a finite group, its dual $\hat G$ is a finite set and we denote the number of the elements in $\hat G$ by $|\hat G|$. 
We use the following notation for any unitary/projective unitary representations $V_1$, $V_2$.
 We write
 $V_1\prec V_2$ if $V_1$ is unitarily equivalent to a sub-representation of $V_2$.
We also say $V_1$ is included in $V_2$ in this case.
Clearly, $\prec$ is a preorder.
We write $V_{1}\cong V_{2}$ if $V_1$ and $V_2$ are unitarily equivalent.

For a unitary representation  (resp. projective unitary representation) of $G$,
we denote by $\bar V$ the complex conjugate representation (resp. projective representation)  of $V$.
(See \cite{simon} section II.6.)

\begin{lem}\label{simple}
There is an $l_{0}\in\nan$
such that for any $l\ge l_{0}$, the tensor product $U^{\otimes l}$
contains any irreducible representation of $G$ as its irreducible component.
\end{lem}
\begin{proof}
Note that the character $\chi_{U}(g)$ is the sum of the eigenvalues of a unitary $U(g)$ acting on
$\bbC^{d}$.
Therefore, the maximal possible value of  $|\chi_{U}(g)|$ is $d$, which is equal to
$\chi_{U}(e)$.
This value is attained only if $U(g)\in\bbT\unit_{\bbC^{d}}$.
By the condition  (\ref{uffl}), for $g\in G\setminus \{e\}$,
$|\chi_{U}(g)|$ is strictly less than $d$.

Now for any irreducible representation $(\bbC^{m},V)$ of  $G$,
for any $l\in\nan$, we have
\begin{align}\label{nncc}
\sum_{g\in G}\overline{\chi_{V}(g)}\chi_{U^{\otimes l}}(g)
=d^{l}\cdot m \lmk
1+\sum_{g\in G\setminus \{e\}}
\frac{\overline{\chi_{V}(g)}}{m}
\lmk \frac{\chi_{U}(g)}{ d}
\rmk^{l}
\rmk.
\end{align}
Note that 
\begin{align}
\frac{\overline{\chi_{V}(g)}}{m}
\lmk \frac{\chi_{U}(g)}{ d}
\rmk^{l}
\end{align}
for $g\in G\setminus \{e\}$
converges to $0$ because of $|\chi_{U}(g)|<d$.
Therefore, for $l$ large enough, 
the left hand side of (\ref{nncc}) is non-zero.
In other word, for $l$ large enough, $V$ is an irreducible component of
$U^{\otimes l}$.
As $\hat G$ is a finite set, this proves the Lemma.
\end{proof}

From this we obtain the following.

\begin{lem}\label{abu}
There is an $l_{0}\in\nan$
such that 
$\alpha\prec \beta\otimes U^{\otimes l}$
holds for any 
$\sigma\in Z^{2}(G,\bbT)$,
$\alpha,\beta\in\pgs$,
and 
$l_{0}\le l\in\nan$.
\end{lem}
\begin{proof}
Let $l_{0}$ be the number given in Lemma \ref{simple}.
For any $\sigma\in Z^{2}(G,\bbT)$ and
$\alpha,\beta\in\pgs$,
$\alpha\otimes \bar \beta$ is a genuine representation of $G$.
Let $V\in \hat G$ be an irreducible component of $\alpha\otimes \bar \beta$.
By Lemma \ref{simple},
this $V$ is realized as an irreducible component of $U^{\otimes l}$ for $l\ge l_{0}$. 
Therefore, for $l\ge l_{0}$, we have
\begin{align}
\sum_{g\in G}\overline{\chi_{\alpha}(g)}\chi_{ \beta\otimes U^{\otimes l}}(g)=
\sum_{g\in G}\overline{\chi_{\alpha}(g)\chi_{\bar \beta}(g)}\chi_{U^{\otimes l}}(g)
\ge
\sum_{g\in G}\overline{\chi_{V}(g)}\chi_{U^{\otimes l}}(g)>0.
\end{align}
This means $\alpha\prec \beta\otimes U^{\otimes l}$.
\end{proof}

\begin{lem}\label{smb}
Let $\sigma\in Z^{2}(G,\bbT)$ be a fixed $2$-cocycle.
For any $m\in\nan$, there exists an $N_{m}^{(\sigma)}\in\nan$ satisfying the following:
For
any projective unitary representation $(\caH,u)$ of $G$ with $2$-cocycle
$\sigma$,
$\alpha\in\caP_{\sigma}$, and $\nan\ni N\ge N_{m}^{(\sigma)}$,
we have
\begin{align}
m\cdot \alpha\prec U^{\otimes N}\otimes u.
\end{align}
(Here $m\cdot \alpha$ denotes the $m$ direct sum of $\alpha$. )
\end{lem}
\begin{proof}
First let us consider the case that 
$\caP_{\sigma}$ consists of a unique element $\alpha\in \caP_{\sigma}$.
Then for any $N\in\nan$, and any projective representation $(\caH,u,\sigma)$, 
the multiplicity of $\alpha$ in $U^{\otimes N}\otimes u$ is $\frac{d^{N}\cdot\dim{\caH}}{n_{\alpha}}$,
which is bigger or equal to a $(\caH,u)$-independent value $\frac{d^{N}}{n_{\alpha}}$.
The claim of Lemma \ref{smb} follows from this immediately for this case.

Next let us consider the case that the number of elements $|\caP_{\sigma}|$ in $\caP_{\sigma}$,
is larger than $1$.
From Lemma \ref{abu}, choose $l_{0}\in \nan$ so that
$\alpha\prec\beta\otimes U^{\otimes l}$
for all $l\ge l_{0}$ and $\alpha,\beta\in\pgs$.
For any $m\in\nan$, choose $M_{m}\in\nan$ so that 
$|\caP_{\sigma}|^{M_{m}}>m$. Here we use the condition that $|\caP_{\sigma}|>1$.
We set $N_{m}^{(\sigma)}:=l_{0}(M_{m}+1)$.
Let $(\caH,u)$ be a projective unitary representation  of $G$ with $2$-cocycle
$\sigma$,
$\alpha\in\caP_{\sigma}$, and $\nan\ni N\ge N_{m}^{(\sigma)}$.
We would like to show that $m\cdot \alpha\prec U^{\otimes N}\otimes u$.
By the choice of $N_{m}^{(\sigma)}$, $N$ can be decomposed as
$N=k_{1}+k_{2}+\cdots+k_{M_{m}}+k_{M_{m}+1}$ with some $l_{0}\le k_{j}\in\nan$, $j=1,\ldots,M_{m}+1$.
For each $j=1,\ldots,M_{m}+1$ and $\beta,\gamma\in\pgs$,
we denote the multiplicity of $\gamma$ in $U^{\otimes k_{j}}\otimes \beta$ 
by $n_{\beta,\gamma}^{(j)}$.
From the choice of $l_{0}$, we have $1\le n_{\beta,\gamma}^{(j)}$
for any $j=1,\ldots,M_{m}+1$ and $\beta,\gamma\in\pgs$.
Fix some $\beta_{0}\in\pgs$ such that $\beta_{0}\prec u$.
From this, we get
\begin{align}
&m\cdot \alpha\prec |\caP_{\sigma}|^{M_{m}}\cdot \alpha
=
\bigoplus_{\gamma_{1},\gamma_{2},\ldots,\gamma_{ {M_{m}}}}\alpha
\prec 
\bigoplus_{\gamma_{1},\gamma_{2},\ldots,\gamma_{ {M_{m}}}, \gamma_{ {M_{m}+1}}}
n_{\beta_0,\gamma_{1}}^{(1)}n_{\gamma_{1},\gamma_{2}}^{(2)}\cdots
n_{\gamma_{{M_{m}}},{\gamma_{ {M_{m}+1}}}}^{( {M_{m}+1})}\cdot \gamma_{ {M_{m}+1}}\nonumber\\
&
\prec U^{\otimes k_{M_{m}+1}}\otimes U^{\otimes k_{M_{m}}}\otimes
\cdots U^{\otimes k_{2}}\otimes U^{\otimes k_{1}}\otimes\beta_0
=
U^{\otimes N}\otimes \beta_0\prec
U^{\otimes N}\otimes u.
\end{align}
This completes the proof.
\end{proof}

Now we are ready to show the main statement of this section.
From the following Lemma, we see that for any $\omega\in SPG(\caA)$,
$u_{\omega,\vp}$ contains all elements of $\caP_{\sigma_{\omega,\vp}}$
with infinite multiplicity.

\begin{thm}\label{zbdip}
Let $\Gamma$ be an infinite subset of $\bbZ$.
Let $(\caL,\rho, u,\sigma)$ be a quadruple such that
\begin{description}
\item[(i)]
 $\rho$ is a $*$-representation of $\caA_{\Gamma}$ on a Hilbert space
 $\caL$,
\item[(ii)]
$u$ is a projective unitary representation  of $G$ on $\caL$
with a $2$-cocycle $\sigma$,
\item[(iii)]
for any $g\in G$, we have
\begin{align}
\rho\circ \tau_{\Gamma}(g)
=\Ad\lmk u(g)\rmk\circ \rho.
\end{align}
\end{description}
Then $u$ contains all elements of $\pgs$ with infinite multiplicity.
\end{thm}
\begin{proof}
Fix any $\alpha\in\caP_{\sigma}$ and $m\in\nan$.
We would like to show that $m\cdot \alpha\prec u$.
Let $N_{m}^{(\sigma)}$ be the number given in Lemma \ref{smb} for this fixed $m$.
Let $\Lambda$ be a subset of $\Gamma$ such that  $|\Lambda|=N_{m}^{(\sigma)}$.
We may factorize $(\caL, \rho, u)$
to $\Lambda$-part and $\Gamma\setminus \Lambda$-part as follows:
There exist a $*$-representation  $(\tilde \caL,\tilde \rho)$ of $\caA_{\Gamma\setminus \Lambda}$ and a projective unitary representation $\tilde u$ of $G$ on $\tilde\caL$
with $2$-cocycle $\sigma$, implementing $\tau_{\Gamma\setminus \Lambda}$.
There exists a unitary $W:\caL\to \lmk \bigotimes_{\Lambda}\bbC^{d}\rmk\otimes \tilde\caL$
such that
\begin{align}\label{prev}
W\rho(a)W^{*}=\lmk \id_{\caA_{\Lambda}}\otimes \tilde\rho\rmk (a),\quad a\in\caA_{\Gamma},
\end{align}
and
\begin{align}\label{up}
W u(g)W^{*}=\lmk \bigotimes_{\Lambda}U(g)\rmk\otimes \tilde u(g),\quad g\in G.
\end{align}

More precisely, set
 $I_{0}=(i_{k})_{k\in\Lambda}\in\{1,\ldots, d\}^{\times \Lambda}$, 
 with $i_{k}=1$ for all $k\in\Lambda$.
 We define the Hilbert space 
 $\tilde\caL$ by $\tilde\caL:=\rho\lmk E_{I_{0},I_{0}}^{(\Lambda)}\rmk\caL$,
 and the $*$-representation $\tilde\rho$ of $\caA_{\Gamma\setminus \Lambda}$ on $\tilde\caL$ by
 \begin{align}
 \tilde\rho\lmk a\rmk :=\rho\lmk  E_{I_{0},I_{0}}^{(\Lambda)}\otimes a\rmk,\quad a\in \caA_{\Gamma\setminus\Lambda}.
 \end{align}
 The unitary $W:\caL\to \lmk \bigotimes_{\Lambda}\bbC^{d}\rmk\otimes \tilde\caL$
 is defined by
 \begin{align}
 W\xi:=\sum_{I\in\{1,\ldots, n\}^{\times \Lambda}}e_{I}^{(\Lambda)}\otimes 
 \rho\lmk E_{I_{0},I}^{(\Lambda)}\rmk\xi,\quad \xi\in\caL.
 \end{align}
 It is straight forward to check (\ref{prev}).
 
 By a straight forward calculation using (\ref{prev}), we can check that
  $\lmk \lmk \bigotimes_{\Lambda}U(g)\rmk^{*}\otimes \unit_{\tilde \caL}\rmk Wu(g)W^{*}$ with $g\in G$
  commute with any element of $\lmk \bigotimes_{\Lambda}\Mat_{d}\rmk\otimes \bbC \unit_{\tilde\caL}$.
 Hence there exists a unitary $\tilde u(g)$ on $\tilde \caL$ such that
  $\lmk\lmk \bigotimes_{\Lambda}U(g)\rmk)^{*}\otimes\unit_{\tilde \caL}\rmk Wu(g)W^{*}=\unit_{\bigotimes_{\Lambda}\bbC^{d}}\otimes
  \tilde u(g)$. This gives (\ref{up}).
  It is straight forward to check that $\tilde u$ is 
  a projective unitary representation of $G$
with $2$-cocycle $\sigma$ implementing $\tau_{\Gamma\setminus \Lambda}$.

From (\ref{up}) and Lemma \ref{smb}, we have
\begin{align}
m\cdot \alpha \prec U^{\otimes  N_{m}^{(\sigma)}}\otimes\tilde u=
U^{\otimes  \Lambda}\otimes\tilde u
\cong u.
\end{align}
This completes the proof.

\end{proof}
Recall Definition \ref{pao}.
We note that $c_{\omega,R}$ and $c_{\omega,L}$
are not independent.
\begin{lem}\label{clmcr}
For any $\omega\in SPG(\caA)$, 
we have $c_{\omega,R}=c_{\omega,L}^{-1}$.
\end{lem}

\begin{proof}
Let $(\caH,\pi,\Omega)$ be the GNS triple of $\omega$.
As $\omega$ satisfies the split property, 
there are
Hilbert spaces $\caH_{L},\caH_{R}$ and a unitary
$W:\caH\to\caH_{L}\otimes \caH_{R}$
such that 
\begin{align}\label{mmd}
W\pi\lmk \caA_{R}\rmk''W^{*}=\bbC\unit_{\caH_{L}}\otimes \caB(\caH_{R}),\quad
W\pi\lmk \caA_{R}\rmk'W^{*}=\caB(\caH_{L})\otimes\bbC\unit_{\caH_{R}}.
\end{align}
(See Theorem 1.31 V \cite{takesaki}.)
From (\ref{mmd}), $\pi\lmk \caA_{L}\rmk''\subset \pi\lmk \caA_{R}\rmk'$
and
 $\pi\lmk \caA_{L}\rmk''\vee \pi\lmk \caA_{R}\rmk''=\caB(\caH)$,
 we obtain
 \begin{align}
 W\pi\lmk \caA_{L}\rmk''W^{*}=\caB(\caH_{L})\otimes\bbC\unit_{\caH_{R}}.
 \end{align}
Hence we obtain irreducible representations $(\caH_{L},\pi_{L})$
and $(\caH_{R},\pi_{R})$ of $\caA_{L}$, $\caA_{R}$
such that 
\begin{align}
&W \pi\lmk
a\otimes b
\rmk W^{*}
=\pi_{L}(a)\otimes\pi_{R}(b)
,\quad a\in\caA_{L},\quad b\in\caA_{R}.
\end{align}
The triple $(\caH_{L}\otimes \caH_{R}, \pi_{L}\otimes\pi_{R}, W\Omega)$
is a GNS triple of $\omega$.
Therefore, $\omega\vert_{\caA_{R}}$ is 
$\pi_{R}$-normal.
As $\pi_{R}(\caA_{R})''$ is a factor,
$\pi_{R}$ and the GNS representation of $\omega\vert_{\caA_{R}}$
are quasi-equivalent.
Similarly, $\pi_{L}$ and the GNS representation of $\omega\vert_{\caA_{L}}$
are quasi-equivalent.

By the $\tau$-invariance of $\omega$,
there is a unitary representation $V$ of $G$ on $\caH_{L}\otimes \caH_{R}$
given by
\begin{align}
V(g)\lmk \pi_{L}\otimes\pi_{R}\rmk (a)W\Omega
=\lmk \pi_{L}\otimes\pi_{R}\rmk \lmk \tau(g)\lmk a\rmk\rmk W\Omega,\quad
g\in G,\quad a\in\caA.
\end{align}
On the other hand, by Proposition \ref{unieq}, there are
projective unitary representations $u_{{L}}$, $u_{{R}}$ of $G$ 
on $\caH_{L}$, $\caH_{R}$ such that
\begin{align}
\pi_{L}\circ \tau_{L}(g)\lmk a\rmk
=\Ad\lmk u_{L}(g)\rmk\circ \pi_{L}\lmk a\rmk,\quad 
\pi_{R}\circ \tau_{R}(g)\lmk b\rmk
=\Ad\lmk u_{R}(g)\rmk\circ \pi_{R}\lmk b\rmk,
\end{align}
for all  $a\in\caA_{L}$,  $b\in\caA_{R}$ and $g\in G$.
Note that 
\begin{align}
\Ad\lmk V(g)\rmk\circ \lmk \pi_{L}\otimes\pi_{R}\rmk(x)
=\lmk \pi_{L}\otimes\pi_{R}\rmk\circ \tau(g)(x)
=\Ad \lmk  u_{L}(g)\otimes  u_{R}(g)\rmk \circ \lmk \pi_{L}\otimes\pi_{R}\rmk(x),
\end{align}
for all $x\in\caA$.
As $\lmk \pi_{L}\otimes\pi_{R}\rmk(\caA)''=\caB(\caH_{L}\otimes \caH_{R})$,
this means that there is a map $b:G\to\bbT$
such that 
\begin{align}\label{uuv}
u_{L}(g)\otimes  u_{R}(g)=b(g) V(g),\quad g\in G.
\end{align}
Let $\sigma_{L},\sigma_{R}\in\bbZ^{2}(G,\bbT)$
be $2$-cocycles of $u_{L}$, $u_{R}$ respectively.
From (\ref{uuv}),
we obtain
\begin{align}
\sigma_{L}\sigma_{R}=\sigma_{b}
\end{align}
 (Here $\sigma_{b}$ is defined by (\ref{bound}).)
 This means
 \begin{align}
 c_{\omega,R}=[\sigma_{R}]_{H^{2}(G,\bbT)}
 =[\sigma_{L}^{-1}]_{H^{2}(G,\bbT)}
 =c_{\omega,L}^{-1}.
 \end{align}

\end{proof}
\section{Twisted $C^{*}$-dynamical system}\label{crossedproduct}
In this section we briefly recall basic facts about twisted $C^{*}$-crossed product. 
Throughout this section, let $\Gamma$ be an infinite subset of $\bbZ$, and
 $\sigma\in Z^{2}(G,\bbT)$. 
 The quadruple $(G, \caA_{\Gamma}, \tau_{\Gamma},\sigma)$ 
 is a twisted $C^{*}$-dynamical system which we denote by $\Sigma_{\Gamma}^{(\sigma)}$.
 (This is a simple version of \cite{bedos}.)

A covariant representation of $\Sigma_{\Gamma}^{(\sigma)}$ is a triple
$(\caH,\pi, u)$ where $\pi $ is a $*$-representation of the $C^{*}$-algebra $\caA_{\Gamma}$ 
on a Hilbert space 
$\caH$
and $u$ is a projective unitary representation of $G$
with $2$-cocycle $\sigma$ on  $\caH$
such that 
\begin{align}
u(g) \pi(a) u(g)^{*}= \pi\lmk\tau_{\Gamma}(g)(a)\rmk,\quad
a\in \caA_{\Gamma},\quad g\in G.
\end{align}
In this paper, we say the covariant representation $(\caH,\pi, u)$ is irreducible if
$\pi $ is an irreducible representation of $\caA_{\Gamma}$.
Note that for a quadruple$(\caL_{\omega,{\vp}}, \rho_{\omega,{\vp}}, u_{\omega,{\vp}},\sigma_{\omega,{\vp}})$ associated to $(\omega\vert_{\caA_{{\vp}}},\tau_{\vp})$
with $\omega\in SPG(\caA)$
(Definition \ref{pao}), $(\caL_{\omega,{\vp}}, \rho_{\omega,{\vp}}, u_{\omega,{\vp}})$
is an irreducible covariant representation of
$\Sigma_{\Gamma_{\vp}}^{(\sigma_{\omega,{\vp}})}$.

Let $C(G,\caA_{\Gamma})$ be 
the linear space of $\caA_{\Gamma}$-valued functions 
on $G$. 
We equip $C(G,\caA_{\Gamma})$ with a product and $*$-operation as follows:
\begin{align}
&f_{1}*f_{2}(h)
:=\sum_{g\in G}\sigma(g,g^{-1}h)\cdot f_{1}(g)\cdot \tg(g)\lmk f_{2}(g^{-1}h)\rmk,\quad h\in G,
\label{multi}\\
&f^{*}(h):=
\overline{\sigma(h^{-1}, h)}\tg\lmk h\rmk \lmk
f(h^{-1})^{*}
\rmk,\quad h\in G,
\end{align}
for $f_{1},f_{2},f\in C(G,\caA_{\Gamma})$.
The linear space $C(G,\caA_{\Gamma})$ which is a $*$-algebra
with these operations is denoted by $C(\Sigma_{\Gamma}^{(\sigma)})$.
We will omit the symbol $*$ for the multiplication (\ref{multi}).

For a covariant representation $(\caH,\pi, u)$ of 
$\Sigma_{\Gamma}^{(\sigma)}$, we may introduce a $*$-representation 
$(\caH,\pi\times u)$ of
$C(\Sigma_{\Gamma}^{(\sigma)})$ by
\begin{align}
\lmk \pi\times u\rmk (f)
:=\sum_{g\in G}\pi\lmk f(g)\rmk u(g),\quad f\in C(\Sigma_{\Gamma}^{(\sigma)}).
\end{align}
The full twisted crossed product of $\Sigma_{\Gamma}^{(\sigma)}$, denoted $C^{*}(\Sigma_{\Gamma}^{(\sigma)})$
is the completion of $C(\Sigma_{\Gamma}^{(\sigma)})$ with respect to the norm
\begin{align}
\lV f\rV_{u}:=
\sup\left\{
\lV\lmk \pi\times u\rmk (f)\rV
\mid
(\pi,u) : \text{covariant representation}
\right\},\quad f\in C(\Sigma_{\Gamma}^{(\sigma)}).
\end{align}
From any representation $(\caH,\pi)$ of $\caA_{\Gamma}$,
we can define a covariant representation $(\caH\otimes l^{2}(G)\simeq l^{2}(G,\caH), \tilde\pi, \tilde u_{\pi})$ of $\Sigma_{\Gamma}^{(\sigma)}$
by
\begin{align}
\lmk \tilde \pi(a) \xi\rmk(g):=\pi\lmk \tg(g^{-1})(a)\rmk\xi(g),\quad
a\in\caA_{\Gamma},\quad \xi\in l^{2}(G,\caH),\quad g\in G,
\end{align}
and $\tilde u_{\pi}=\unit_{\caH}\otimes u^{\sigma}_{r}$.
Here, $u^{\sigma}_{r}$ is a projective unitary representation 
with $2$-cocycle $\sigma$ on
$l^{2}(G)$ defined by
\begin{align}
\lmk u^{\sigma}_{r}(g)\xi\rmk (h)
=\sigma(g, g^{-1}h)\xi(g^{-1}h),\quad g,h\in G,\quad \xi\in l^{2}(G).
\end{align}
Note that $\pi$ is faithful because $\caA_{\Gamma}$ is simple. Therefore, the representation
$ \tilde\pi\times  \tilde u_{\pi}$ of $C(\Sigma_{\Gamma}^{(\sigma)})$ given by
 $(\caH\otimes l^{2}(G)\simeq l^{2}(G,\caH), \tilde\pi, \tilde u_{\pi})$ 
is faithful.
 We define a $C^{*}$-norm $\lV\cdot\rV_{r}$ on $C(\Sigma_{\Gamma}^{(\sigma)})$ by
 \begin{align}
 \lV f\rV_{r}:=
 \lV
 \tilde\pi\times  \tilde u_{\pi}(f)
 \rV_{\caB(\caH\otimes l^{2}(G))},\quad f\in C(\Sigma_{\Gamma}^{(\sigma)}).
 \end{align}
 The completion $C_{r}^{*}(\Sigma_{\Gamma}^{(\sigma)})$ 
 of $C(\Sigma_{\Gamma}^{(\sigma)})$ with respect to this norm is the reduced twisted crossed product of $\Sigma_{\Gamma}^{(\sigma)}$.
 As we are considering a finite group $G$, we have
 $C(\Sigma_{\Gamma}^{(\sigma)})=C_{r}^{*}(\Sigma_{\Gamma}^{(\sigma)})=C^{*}(\Sigma_{\Gamma}^{(\sigma)})$, and
 $\lV\cdot \rV_{r}=\lV\cdot\rV_{u}$.
 
 For each $a\in \caA_{\Gamma}$, $\xi_{a}:G\ni g\mapsto \delta_{g,e}a\in \caA_{\Gamma}$ defines an element of $C^{*}(\Sigma_{\Gamma}^{(\sigma)})$. 
 The map $\xi:\caA_{\Gamma}\ni a\mapsto \xi_{a}\in C^{*}(\Sigma_{\Gamma}^{(\sigma)})$
 is a unital faithful $*$-homomorphism.
 Note that $\xi_{\unit_{\caA_{\Gamma}}}$ is the identity of $C^{*}\lmk\Sigma_{\Gamma}^{(\sigma)}\rmk$.
 Hence 
 the $C^{*}$-algebra $\caA_{\Gamma}$ can be regarded as a subalgebra
 of $C(\Sigma_{\Gamma}^{(\sigma)})=C_{r}^{*}(\Sigma_{\Gamma}^{(\sigma)})=C^{*}(\Sigma_{\Gamma}^{(\sigma)})$.
 Therefore, we simply write $a$ to denote $\xi_{a}$.
 
 From the condition (\ref{uffl}), 
 for any $g\in G$ with $g\neq e$, the automorphism $\tau_{\Gamma}(g)$ is properly outer.
  Therefore, by the argument in \cite{elliott} Theorem 3.2, $C(\Sigma_{\Gamma}^{(\sigma)})=C_{r}^{*}(\Sigma_{\Gamma}^{(\sigma)})=C^{*}(\Sigma_{\Gamma}^{(\sigma)})$ is simple.
 
 As $\caA_{\Gamma}$ is unital, we have unitaries $\lambda_{g}\in C^{*}(\Sigma_{\Gamma}^{(\sigma)})$, $g\in G$,
 defined by $G\ni h\mapsto \delta_{g,h}\unit_{\caA_{\Gamma}}\in \caA_{\Gamma}$ such that
 \begin{align}\label{ldef}
 &\lambda_{g}\lambda_{h}=\sigma(g,h)\lambda_{gh},\quad g,h\in G,\nonumber\\
 &\lambda_{g}a \lambda_{g}^{*}=\tau_{\Gamma}(g)\lmk a\rmk,\quad a\in \caA_{\Gamma},\; g\in G.
 \end{align}
 Note that $\lambda_{e}=\xi_{\unit_{\caA_{\Gamma}}}$ is the identity of $C^{*}\lmk\Sigma_{\Gamma}^{(\sigma)}\rmk$.


We set
\begin{align}
C^{*}\lmk\Sigma_{\Gamma}^{(\sigma)}\rmk^{G}
:=\left\{
f\in C^{*}\lmk\Sigma_{\Gamma}^{(\sigma)}\rmk\mid
\Ad\lmk\lambda_{g}\rmk(f)=f,\quad g\in G
\right\}.
\end{align}

Let $(\caH,\pi, u)$ be an irreducible covariant representation 
of
$\Sigma_{\Gamma}^{(\sigma)}$.
The projective unitary representation $u$ has an irreducible decomposition given by some Hilbert spaces $\{\caK_{\gamma}\mid \gamma\in \caP_{\sigma}\}$
(Lemma \ref{pd} and Notation \ref{ddcom}). Namely we have
\begin{align}\label{bvp}
u(g)=\bigoplus_{\alpha\in \pgs}
\uu_{\alpha}(g)\otimes \unit_{\caK_{\alpha}},\quad g\in G,\quad \text{and}\quad
u(G)'=
\bigoplus_{\alpha\in \pgs}
\unit_{\caH_{\alpha}}\otimes \caB\lmk {\caK_{\alpha}}\rmk.
\end{align}
Note that 
\begin{align}\label{rg}
\lmk \pi\times u\rmk(\lambda_{g})=u(g), \quad g\in G .
\end{align}
From this we have
\begin{align}
\lmk \pi\times u\rmk
\lmk C^{*}\lmk\Sigma_{\Gamma}^{(\sigma)}\rmk^{G}\rmk
\subset u(G)'=
\bigoplus_{\alpha\in \pgs}
\bbC\unit_{\caH_{\alpha}}\otimes \caB\lmk {\caK_{\alpha}}\rmk.
\end{align}
%

The following proposition is the immediate consequence of Theorem \ref{zbdip}.
\begin{prop}\label{zbdi}
Let  $(\caH,\pi,u)$ be an irreducible covariant representation of 
$\Sigma_{\Gamma}^{(\sigma)}$.
Then $u$ contains
all elements of $\pgs$ with infinite multiplicity.
\end{prop}

\section{Homogeneity}\label{homogeneity}
Throughout this section we fix $\Gamma=\Gamma_{L},\Gamma_{R}$, $\sigma\in Z^{2}(G,\bbT)$, and
$\alpha\in \pgs$.
We use the following notation.
\begin{notation}\label{nagai}Let $(\caH,\pi,u)$ be an irreducible covariant representation of 
$\Sigma_{\Gamma}^{(\sigma)}$ with
an irreducible decomposition of $u$ given by a set of Hilbert spaces 
$\{\caK_{\gamma}\mid \gamma \in  \pgs\}$. We use the symbol $\hat\pi$ to denote the irreducible representation 
\begin{align}\label{phat}
\hat\pi=\id_{\caB(\caH_{\alpha})}\otimes \lmk\pi\times u\rmk
\end{align}
of
$\caB(\caH_{\alpha})\otimes C^{*}(\Sigma_{\Gamma}^{(\sigma)})$ on $\caH_{\alpha}\otimes \caH$.

For a unit vector $\xi\in \caK_{\alpha}$,
we may define a state $\hat\varphii_{\xi}$ on $\caB(\caH_{\alpha})\otimes C^{*}(\Sigma_{\Gamma}^{(\sigma)})$
by
\begin{align}\label{nekosan}
\hat\varphii_{\xi}(x):=
\braket{\tilde\xi}
{\hat\pi
\lmk x\rmk \tilde\xi},\quad x\in \caB(\caH_{\alpha})\otimes C^{*}(\Sigma_{\Gamma}^{(\sigma)}).
\end{align}
Here, $\tilde \xi$ is an element of $ \caH_{\alpha}\otimes \caH$,
\begin{align}\label{kaeru}
\tilde\xi:=\Omega_{\alpha}\otimes \xi\in \caH_{\alpha}\otimes \caH_{\alpha}\otimes \caK_{\alpha}
\hookrightarrow 
 \caH_{\alpha}\otimes \caH,
\end{align}
regarding $ \caH_{\alpha}\otimes \caH_{\alpha}\otimes \caK_{\alpha}$ as a subspace of $ \caH_{\alpha}\otimes \caH$. (See Notation \ref{ddcom}.)
Recall that $\Omega_{\alpha}$ is defined in  (\ref{oa}).
We call this $\hat\varphii_{\xi}$ a state on $\caB(\caH_{\alpha})\otimes C^{*}(\Sigma_{\Gamma}^{(\sigma)})$
given by $(\caH,\pi,u,\xi)$.
By the irreducibility of $\pi$, $\hat\pi$ is irreducible and $\hat \varphii_{\xi}$ is a pure state on $\caB(\caH_{\alpha})\otimes C^{*}(\Sigma_{\Gamma}^{(\sigma)})$.
Note that $(\caH_{\alpha}\otimes \caH,\hat\pi,\tilde \xi)$
is a GNS triple of $\hat\varphii_{\xi}$.

\end{notation}

The goal of this section is to prove the following Proposition.

\begin{prop}\label{lem6}
 Let $(\caH_{i},\pi_{i},u_{i})$ with $i=0,1$ be
irreducible  covariant representations of $\Sigma_{\Gamma}^{(\sigma)}$
with irreducible decomposition of $u_{i}$ given by a set of Hilbert spaces 
$\{\caK_{\gamma,i}\mid \gamma \in  \pgs\}$.
Let 
$\xi_{i}\in\caK_{\alpha,i}$ be unit vectors in $\caK_{\alpha,i}$ for $i=0,1$.
(Recall Proposition \ref{zbdi} for existence of such vectors.)
Let $\hat\varphii_{{\xi}_{i}}$ be a state on $\caB(\caH_{\alpha})\otimes C^{*}(\Sigma_{\Gamma}^{(\sigma)})$
given by $(\caH_{i},\pi_{i},u_{i},\xi_{i})$, for each $i=0,1$.
Let
$\varphii_{i}$ be the restriction of $\hat\varphii_{{\xi}_{i}}$
onto $\caA_{\Gamma}$.
 Then there exists 
 a norm-continuous path $w:[0,\infty)\to \caU(\caA_{\Gamma}^{G})$
 with $w(0)=\unit$,
 such that  
 \begin{enumerate}
 \item
for each $a\in\caA_{\Gamma}$, the limit 
 \begin{align}
 \lim_{t\to\infty}\Ad\lmk w(t)\rmk(a)=:\Xi_{\Gamma}(a)
 \end{align}
exists and defines 
 an automorphism $\Xi_{\Gamma}$ on $\caA_{\Gamma}$, and
 \item the automorphism $\Xi_{\Gamma}$ in {\it 1.}
 satisfies $\varphii_{1}=\varphii_{0}\circ\Xi_{\Gamma}$.
 \end{enumerate}
\end{prop}

\begin{rem}
Basically, what we would like to do is to connect some $\pi_{0}$-normal state
$\varphi_{0}$ and some $\pi_{1}$-normal state
$\varphi_{1}$ via some $\Xi_{\Gamma}\in \ainn^{G} (A_{\Gamma})$.
Without symmetry, $\varphi_{0}$ and 
$\varphi_{1}$ can be taken to be pure states.
When the symmetry comes into the game, to guarantee that $\Xi_{\Gamma}$
commutes with $\tau_{\Gamma}(g)$,
we would like to assume that $\varphi_{0}$ and $\varphi_{1}$ 
are $\tau_{\Gamma}$-invariant.
If $\sigma$ is trivial, there is a $u_{i}$-invariant non-zero vector that we may 
find such pure $\tau_{\Gamma}$-invariant states $\varphi_{0}$ and 
$\varphi_{1}$.
But if the cohomology class of $\sigma$ is not trivial,
$u_{i}$ does not have a non-zero invariant vector.
However, there is still a rank $n_{\alpha}$ $u_{i}$-invariant
density matrix.
That is the reason why we consider $\caB(\caH_{\alpha})\otimes C^{*}(\Sigma_{\Gamma}^{(\sigma)})$.
Note that the density matrix of $\varphi_{i}$
is a rank $n_{\alpha}$ operator which commutes with $u_{i}$.
\end{rem}

For the proof of Proposition \ref{lem6}, we use the machinery used in \cite{fkk} and \cite{kos}.
(See Appendix \ref{kosfkk}.)
However, as we would like to have a path in the fixed point algebra $\caA_{\Gamma}^{G}$,
we need additional arguments.
For that purpose, the following Lemma plays an important role.

\begin{lem}\label{lem11iii}
Let $\Gamma_{0}$ be an infinite subset of $\bbZ$.
Let $(\caH,\pi,u)$ be an irreducible covariant representation of 
$\Sigma_{{\Gamma_0}}^{(\sigma)}$ with
an irreducible decomposition of $u$ given by a set of Hilbert spaces 
$\{\caK_{\gamma}\mid \gamma \in  \pgs\}$.
Then there exist irreducible $*$-representations $(\caK_{\gamma},\pi_{\gamma} )$, $\gamma\in \caP_{\sigma}$
of $\caA_{{\Gamma_0}}^{G}$
such that 
\begin{align}\label{eq53}
\pi(a)=\bigoplus_{\gamma\in \pgs} \unit_{\caH_{\gamma}}\otimes \pi_{\gamma}(a),\quad
a\in \caA^{G}_{{\Gamma_0}}.
\end{align}
Furthermore,
we have
\begin{align}\label{ttg}
\pi\lmk\caA^{G}_{{\Gamma_0}}\rmk''=\bigoplus_{\gamma\in \pgs}\unit_{\caH_{\gamma}}\otimes
B(\caK_{\gamma}).
\end{align}
\end{lem}
\begin{notation}\label{agrep}
We call $\left\{(\caK_{\gamma}, \pi_{\gamma})\mid\gamma\in \pgs\right\}$,
the family of representations of $\caA^{G}_{{\Gamma_0}}$ associated to $(\caH,\pi,u)$.
\end{notation}
\begin{proof}
For any $a\in\caA_{{\Gamma_0}}^{G}$, we have $\pi(a)\in u(G)'$.
Therefore, from Lemma \ref{pd},
each $\pi(a)$ with $a\in\caA_{{\Gamma_0}}^{G}$
has a form 
\begin{align}
\pi(a)=\bigoplus_{\gamma\in \pgs} \unit_{\caH_{\gamma}}\otimes \pi_{\gamma}(a),
\end{align}
with uniquely defined $\pi_{\gamma}(a)\in\caB(\caK_{\gamma})$,
for each $\gamma\in\caP_{\sigma}$.

As $\pi$ is a $*$-representation, for each $\gamma\in \pgs$, the map
$\pi_{\gamma}:\caA^{G}_{{\Gamma_0}}\ni a\mapsto \pi_{\gamma}(a)\in \caB(\caK_{\gamma})$ is 
a $*$-representation and we have
\begin{align}
\pi\lmk\caA^{G}_{{\Gamma_0}}\rmk''\subset\bigoplus_{\gamma\in \pgs}\unit_{\caH_{\gamma}}\otimes
B(\caK_{\gamma}).
\end{align}

We claim that each $\pi_{\gamma}$ is an irreducible representation of $\caA^{G}_{{\Gamma_0}}$, and
(\ref{ttg}) holds.
To see this, note that for any $x\in \caB(\caK_{\gamma})$, 
there exists a bounded net $\{a_{\lambda}\}_{{\lambda}}\in\caA$ such that
$\pi(a_{\lambda})$ converges to $\unit_{\caH_{\gamma}}\otimes x\in
\caB(\caH_{\gamma})\otimes \caB(\caK_{\gamma})\subset \caB(\caH)$ 
in the $\sigma$-strong topology,
by the irreducibility of $\pi$ and the Kaplansky density theorem.
For this $\{a_{\lambda}\}_{{\lambda}}$,
we have
\begin{align}\label{access}
\frac 1{|G|} \sum_{g\in G} u_{g} \pi(a_{\lambda}) u_{g}^{*}
=\pi\lmk
\frac 1{|G|} \sum_{g\in G}\tau_{{\Gamma_0}}(g)\lmk a_{\lambda}\rmk
\rmk
=
\bigoplus_{\gamma\in \pgs} \unit_{\caH_{\gamma}}\otimes \pi_{\gamma}\lmk
\frac 1{|G|} \sum_{g\in G}\tau_{{\Gamma_0}}(g)\lmk a_{\lambda}\rmk\rmk
\in \pi\lmk\caA^{G}_{{\Gamma_0}}\rmk''
\end{align}
because $\frac 1{|G|} \sum_{g\in G}\tau_{{\Gamma_0}}(g)\lmk a_{\lambda}\rmk$
is $\tau_{{\Gamma_0}}$-invariant.
Since the left hand side of (\ref{access}) converges to
$\unit_{\caH_{\gamma}}\otimes x$ 
in the $\sigma$-strong topology,
we conclude that
$\unit_{\caH_{\gamma}}\otimes x\in \pi\lmk\caA^{G}_{{\Gamma_0}}\rmk''$. 
Hence (\ref{ttg}) holds. Looking at the $\gamma$-component of (\ref{access}),
we see that $\pi_{\gamma}\lmk\caA^{G}_{{\Gamma_0}}\rmk''=\caB(\caK_{\gamma})$.
Hence $\pi_{\gamma}$ is irreducible.
This completes the proof.
\end{proof}

For each Lemma below, we use the machinery used in \cite{fkk} and \cite{kos}.
We remark arguments required to get a path inside of $\caA_{\Gamma}^{G}$.
\begin{lem}\label{lem2}
Let $(\caH_{i},\pi_{i},u_{i})$ with $i=0,1$ be irreducible covariant representations of 
$\Sigma_{\Gamma}^{(\sigma)}$ with
irreducible decomposition of $u_{i}$ given by a set of Hilbert spaces 
$\{\caK_{\gamma,i}\mid \gamma \in  \pgs\}$, $i=0,1$.
Let $\xi_{i}\in\caK_{\alpha,i}$, $i=0,1$ be unit vectors.
Let $\hat\varphii_{{\xi}_{i}}$ be a state on $\caB(\caH_{\alpha})\otimes C^{*}(\Sigma_{\Gamma}^{(\sigma)})$
given by $(\caH_{i},\pi_{i},u_{i},\xi_{i})$, for each $i=0,1$.
(Recall Notation \ref{nagai}.)
Then for any $\caF\Subset  \caB(\caH_{\alpha})\otimes C^{*}(\Sigma_{\Gamma}^{(\sigma)})$ and any $\varepsilon>0$, 
there exists a self-adjoint $h\in \caA_{\Gamma}^{G}$
such that
\begin{align}\label{lem2main}
\lv
\hat\varphii_{{\xi}_{0}}\lmk x\rmk-\hat\varphii_{{\xi}_{1}}\circ \Ad(e^{ih})(x)
\rv<\varepsilon,\quad x\in\caF. 
\end{align}
\end{lem}
\begin{proof}
First we prepare some notations.
We denote by $\tilde\xi_{i}$, $
\hat\pi_{i}$ 
the vector and the representation $\tilde\xi$, $
\hat\pi$
defined in Notation \ref{nagai}
with $(\caH,\pi,u,\xi)$ replaced by $(\caH_{i},\pi_{i},u_{i},\xi_{i})$.
(See (\ref{kaeru}) and (\ref{phat}).)
The triple $( \caH_{\alpha}\otimes \caH_{i},\hat\pi_{i}, \tilde\xi_{i})$ is a GNS-triple 
of $\hat\varphii_{{\xi}_{i}}$.
As $ \caB(\caH_{\alpha})\otimes C^{*}(\Sigma_{\Gamma}^{(\sigma)})$ is simple, kernel of 
$\hat\pi_{i}$ is zero for each $i=0,1$.
For each $k,j=1,\ldots, n_{\alpha}$, 
we define an element
\begin{align}
Q_{k,j}^{(\alpha)}:=\frac{n_{\alpha}}{|G|}\sum_{g\in G}
\overline
{\braket{\psi_{k}^{(\alpha)}}{V_{\alpha}(g)\psi_{j}^{(\alpha)}}}
\lambda_{g}
\in C^{*}(\Sigma_{\Gamma}^{(\sigma)}),
\end{align}
with $\lambda_{g}\in C^{*}(\Sigma_{\Gamma}^{(\sigma)})$
introduced in section \ref{crossedproduct} (\ref{ldef}).
We also set
\begin{align}
R^{(\alpha)}:=\frac1{n_{\alpha}}\sum_{k,j=1}^{n_{\alpha}} \ketbra{\psi_{k}^{(\alpha)}}{\psi_{j}^{(\alpha)}}\otimes  Q_{k,j}^{(\alpha)}
\in\caB(\caH_{\alpha})\otimes C^{*}(\Sigma_{\Gamma}^{(\sigma)}).
\end{align}
We claim
\begin{align}\label{puq}
\lmk\pi_{i}\times u_{i}\rmk\lmk Q_{k,j}^{(\alpha)}\rmk
=\ketbra{\psi_{k}^{(\alpha)}}{\psi_{j}^{(\alpha)}}\otimes\unit_{\caK_{\alpha,i}}
\in\caB\lmk  \caH_{\alpha}\otimes \caK_{\alpha,i}\rmk
\subset \caB\lmk \caH_{i}\rmk,
\end{align}
for each $k,j=1,\ldots, n_{\alpha}$ and $i=0,1$.
From this we have
\begin{align}\label{66}
\hat\pi_{i}\lmk R^{(\alpha)}\rmk
=\ketbra{\Omega_{\alpha}}{\Omega_{\alpha}}\otimes \unit_{\caK_{\alpha,i}}
=:P^{(\alpha,i)},\quad i=0,1.
\end{align}
From  (\ref{66}), we obtain 
\begin{align}\label{or1}
\hat\varphii_{{\xi}_{i}}(R^{(\alpha)})=1,\quad i=0,1.
\end{align}
To see (\ref{puq}), recall the orthogonality relation (\ref{orthog}), 
the irreducible decomposition of $u_{i}$ given by $\{\caK_{\gamma,i}\mid \gamma \in  \pgs\}$ and 
that $(\pi_{i}\times u_{i})(\lambda_{g})=u_{i}(g)$ (\ref{rg}).
Then we have
\begin{align}
&\lmk\pi_{i}\times u_{i}\rmk\lmk Q_{k,j}^{(\alpha)}\rmk
=\frac{n_{\alpha}}{|G|}\sum_{g\in G}
\overline
{\braket{\psi_{k}^{(\alpha)}}{V_{\alpha}(g)\psi_{j}^{(\alpha)}}}
u_{i}(g)\nonumber\\
&=\frac{n_{\alpha}}{|G|}\sum_{g\in G}
\overline
{\braket{\psi_{k}^{(\alpha)}}{V_{\alpha}(g)\psi_{j}^{(\alpha)}}}
\lmk
\bigoplus_{\gamma\in \pgs}
\uu_{\gamma}(g)\otimes \unit_{\caK_{\gamma,i}}
\rmk
=\bigoplus_{\gamma\in \pgs}\lmk
\delta_{\alpha,\gamma}
\ketbra{\psi_{k}^{(\alpha)}}{\psi_{j}^{(\alpha)}}\otimes\unit_{\caK_{\alpha,i}}\rmk.
\end{align}

We now start the proof of Lemma.
We fix an arbitrary $\caF\Subset  \caB(\caH_{\alpha})\otimes C^{*}(\Sigma_{\Gamma}^{(\sigma)})$ and $\varepsilon>0$.
We then choose $0<\tilde\varepsilon$ small enough so that 
\begin{align}\label{ted}
\tilde\varepsilon< \min\{1,\frac\varepsilon 2\},\quad
\text{and}\quad 
4\max_{a\in\caF}\lV a\rV 
\tilde\varepsilon^{\frac 12}<\frac\varepsilon 2.
\end{align}
We also set
\begin{align}\label{frset}
\tilde\caF:=
\caF\cup
\left\{
R^{(\alpha)}
\right\}
\Subset B(\caH^{(\alpha)})\otimes  C^{*}(\Sigma_{\Gamma}^{(\sigma)}).
\end{align}
Applying Lemma \ref{kos1312} to this $\tilde\varepsilon$ and $\tilde\caF$,
and pure states $\hat\varphii_{{\xi}_{i}}$, $i=0,1$ of a simple unital $C^{*}$-algebra $ B(\caH^{(\alpha)})\otimes  C^{*}(\Sigma_{\Gamma}^{(\sigma)})$,
we obtain
an $f\in \lmk B(\caH^{(\alpha)})\otimes  C^{*}(\Sigma_{\Gamma}^{(\sigma)})\rmk_{+,1}$
and a unit vector $\zeta\in \caH^{(\alpha)}\otimes \caH_{1}$ such that 
\begin{align}\label{zeta}
\hat\pi_{1}(f)\zeta=\zeta,\quad 
\lV
f\lmk a-\hat\varphii_{{\xi}_{0}}(a)\unit \rmk f
\rV<\tilde\varepsilon,\quad \text{for all}\quad a\in \tilde\caF.
\end{align}
For $P^{(\alpha,1)}$ in (\ref{66}) and the $\zeta$ in (\ref{zeta}), we have
\begin{align}
&\lV
\lmk \unit-P^{(\alpha,1)}\rmk\zeta
\rV^{2}
=\braket{\zeta}{\hat\pi_{1}(f)\lmk \unit- \hat\pi_{1}(R^{(\alpha)})\rmk  \hat\pi_{1}(f)\zeta }
=\braket{\zeta}{\hat\pi_{1}(f)\lmk \hat\varphii_{{\xi}_{0}}(R^{(\alpha)})\unit- \hat\pi_{1}(R^{(\alpha)})\rmk  \hat\pi_{1}(f)\zeta }\nonumber\\
&\le \lV
f\lmk \hat\varphii_{{\xi}_{0}}(R^{(\alpha)})\unit- R^{(\alpha)}\rmk f
\rV
<\tilde \varepsilon.
\end{align}
Here we used (\ref{or1}), for the second equality.
For the inequality we used (\ref{zeta}) and $R^{(\alpha)}\in\tilde \caF$ (\ref{frset}).
Therefore, $P^{(\alpha,1)}\zeta$ is not zero, and we may define a unit vector
\begin{align}
\tilde\zeta:=
\frac 1{\lV P^{(\alpha,1)}\zeta\rV}P^{(\alpha,1)}\zeta\in \caH_{\alpha}\otimes \caH_{1}.
\end{align}
Furthermore, it
satisfies
\begin{align}
\lV \zeta-\tilde\zeta\rV\le
2\tilde\varepsilon^{\frac12}.
\end{align}
From this and two properties in (\ref{zeta})
for any $a\in \caF$, we have
\begin{align}\label{est6}
&\lv
\hat\varphii_{{\xi}_{0}}(a)-\braket{\tilde\zeta}{ \hat\pi_{1}(a)   \tilde\zeta}
\rv
\le
\lv
\braket{\hat\pi_{1}(f)\zeta}{\lmk \hat\varphii_{{\xi}_{0}}(a)- \hat\pi_{1}(a)\rmk  \hat\pi_{1}(f)\zeta }
\rv
+\lv
\braket{\zeta}{\hat\pi_{1}(a)  \zeta }
-\braket{\tilde\zeta}{\hat\pi_{1}(a)\tilde\zeta }
\rv\nonumber\\
&\le
\lV
f\lmk a-\hat\varphii_{{\xi}_{0}}(a)\unit \rmk f
\rV
+2\max_{a\in\caF}\lV a\rV \lV \zeta-\tilde\zeta\rV
<\tilde\varepsilon+2\max_{a\in\caF}\lV a\rV 2\tilde\varepsilon^{\frac12}<\varepsilon,
\end{align}
by the choice of $\tilde\varepsilon$ (\ref{ted}).

Since $ P^{(\alpha,1)}\tilde\zeta= \tilde\zeta$,
there exists a unit vector $\eta\in\caK_{\alpha,1}$ such that
$\tilde\zeta=\Omega_{\alpha}\otimes \eta$.
By Lemma \ref{lem11iii}, 
for each $\gamma\in\pgs$, there exists an irreducible representation $ \pi_{\gamma,1}$
of $\caA_{\Gamma}^{G}$ on $\caK_{\gamma,1}$ 
such that
\begin{align}
\pi_{1}(a)=\bigoplus_{\gamma\in \pgs} \unit_{\caH_{\gamma}}\otimes \pi_{\gamma,1}(a),\quad
a\in \caA^{G}_{\Gamma}.
\end{align}
Applying the Kadison transitivity theorem for unit vectors $\xi_{1}, \eta\in \caK_{\alpha,1}$ and an irreducible representation $(\caK_{\alpha,1}, \pi_{\alpha,1})$ of $\caA_{\Gamma}^{G}$,
we obtain a self-adjoint $h\in \caA^{G}_{\Gamma}$ such that
$\pi_{\alpha,1}(e^{-ih})\xi_{1}=\eta$.
With this $h$, we can write $\tilde\zeta$ as
\begin{align}
\tilde\zeta=\hat\pi_{1}\lmk
\unit_{B(\caH_{\alpha})}\otimes e^{-ih}
\rmk\tilde\xi_{1}.
\end{align}
Hence we obtain
\begin{align}
\hat\varphii_{{\xi}_{1}}\circ \Ad\lmk e^{ih}\rmk
=\braket{\tilde\zeta}{\hat\pi_{1}\lmk \cdot \rmk \tilde\zeta}.
\end{align}
Combining this with (\ref{est6}), we see that (\ref{lem2main}) holds.
\end{proof}
\begin{rem}
The main difference of the proof of Lemma \ref{lem2} from \cite{kos}, \cite{fkk}
is that in order to find $h$ in $\caA_{\Gamma}^{G}$, we add
$R^{(\alpha)}$ to $\caF$.
This allows us to replace $\zeta$ with $\tilde\zeta=\Omega_{\alpha}\otimes \eta$.
From this combined with Lemma \ref{lem11iii}, the problem is reduced to the Kadison transitivity
for the irreducible $(\caK_{\alpha,1},\pi_{\alpha,1}(\caA_{\Gamma}^{G}))$.
Note that $R^{(\alpha)}$ belongs to $B(\caH_{\alpha})\otimes C^{*}(\Sigma_{\Gamma}^{(\sigma)})$
but not in $\caA_{\Gamma}$.
By extending the $C^{*}$-algebra we consider,
we are allowed to have the projection $P^{(\alpha,i)}$ (\ref{66})
corresponding to the irreducible component of $u_{i}$  in the $C^{*}$-algebra.
\end{rem}

\begin{notation}\label{gca}
For $\Lambda\Subset \Gamma$,
we introduce a finite subset of $B(\caH_{\alpha})\otimes C^{*}(\Sigma_{\Gamma}^{(\sigma)})$
given by
\begin{align}
\caG_{\Lambda}:=
\left\{
\frac 1{\sqrt{|G|}} \ketbra{\psi_{j}^{(\alpha)}}{\psi_{1}^{(\alpha)}}\otimes
\lambda_{g} E_{I,I_{0}}^{(\Lambda)}\quad
\mid \quad
j=1,\ldots, n_{\alpha},\; I\in \{1,\ldots,d\}^{\times \Lambda},\;\;
g\in G
\right\}.
\end{align}
Here, we set $I_{0}:=(i_{k})_{k\in\Lambda}\in \{1,\ldots,d\}^{\times \Lambda}$, with $i_{k}=1$ for all $k\in\Lambda$.
\end{notation}
\begin{notation}\label{cond1}
We say an irreducible  covariant representation $(\caH,\pi,u)$ of $\Sigma_{\Gamma}^{(\sigma)}$
and unit vectors $\xi,\eta\in \caH_{\alpha}\otimes\caH$ satisfy
{\it Condition 1}  for a pair $\delta>0$, $\Lambda\Subset \Gamma$,
if the representation $\hat\pi:=\id_{\caB(\caH_{\alpha})}\otimes \lmk \pi\times u\rmk$
of $B(\caH_{\alpha})\otimes C^{*}(\Sigma_{\Gamma}^{(\sigma)})$ satisfies the following:
\begin{enumerate}
\item
For any $x,y\in \caG_{\Lambda}$,
$\hat\pi(x)^{*}\xi$ and $\hat\pi(y)^{*}\eta$ are orthogonal.
\item For any $x,y\in \caG_{\Lambda}$,
\begin{align}
\lv
\braket{\xi}{\hat\pi(x y^{*})\xi}
-\braket{\eta}{\hat\pi(x y^{*})\eta}
\rv<\delta.
\end{align}
\end{enumerate}

\end{notation}
Let $\delta_{\rk}$ be the function given in Lemma \ref{lem6kos}.

\begin{lem}\label{lem4}
For any $\varepsilon>0$ and 
$\Lambda\Subset \Gamma$,
there exists a $\delta_{\nh}(\varepsilon, \Lambda)>0$ 
satisfying the following:
For any irreducible  covariant representation $(\caH,\pi,u)$ of $\Sigma_{\Gamma}^{(\sigma)}$
and unit vectors $\xi,\eta\in \caH_{\alpha}\otimes\caH$ satisfying
{\it Condition 1}  for a pair $\delta_{\nh}(\varepsilon, \Lambda)>0$, $\Lambda\Subset \Gamma$,
there exists a positive element $h$ of $(\caA_{\Gamma\setminus\Lambda}^{G})_{1}$
such that
\begin{align}
\lV
e^{i\pi\hat\pi (h)}\xi-\eta
\rV<
\frac{1}{4\sqrt 2}\delta_{\rk}\lmk \frac\varepsilon 8\rmk.
\end{align}
\end{lem}
\begin{proof}
Recall Lemma \ref{lem28}. We set
\begin{align}
\delta_{\nh}(\varepsilon,\Lambda):=\delta_{\nhh}\lmk\varepsilon, n_{\alpha} d^{|\Lambda|} |G|\rmk,
\end{align}
with $\delta_{\nhh}(\cdot,\cdot)$ in Lemma \ref{lem28}.
We prove that this $\delta_{\nh}$ satisfies the condition above.

 Let us consider an arbitrary irreducible  covariant representation $(\caH,\pi,u)$ of $\Sigma_{\Gamma}^{(\sigma)}$
and unit vectors $\xi,\eta\in \caH_{\alpha}\otimes\caH$ satisfying
{\it Condition 1}  for a pair $\delta_{\nh}(\varepsilon,\Lambda)>0$, $\Lambda\Subset \Gamma$.
We again use the notation $\hat\pi$
(\ref{phat}) for this $\pi$.

We apply Lemma \ref{lem28}, to an infinite dimensional  Hilbert space
$\caH_{\alpha}\otimes \caH$, a unital
$C^{*}$-algebra 
$\hat\pi\lmk \caB(\caH_{\alpha})\otimes C^{*}(\Sigma_{\Gamma}^{(\sigma)})\rmk$
acting irreducibly on $\caH_{\alpha}\otimes \caH$, 
a finite subset $\hat\pi(\caG_{\Lambda})$
of $\hat\pi\lmk \caB(\caH_{\alpha})\otimes C^{*}(\Sigma_{\Gamma}^{(\sigma)})\rmk$, and
unit vectors $\xi,\eta\in \caH_{\alpha}\otimes\caH$.
Note that $\sum_{x\in\hat\pi(\caG_{\Lambda}) }xx^{*}=\unit$
by the definition of $\caG_{\Lambda}$.
From {\it Condition 1}, $\xi,\eta$ satisfy the required conditions in Lemma \ref{lem28}.
By Lemma \ref{lem28} , 
there exists a positive $\tilde h\in \lmk\caB(\caH_{\alpha})\otimes  C^{*}(\Sigma_{\Gamma}^{(\sigma)})\rmk_{+,1}$
such that 
\begin{align}\label{pmm}
\lV
\hat\pi\lmk \bar h\rmk\lmk\xi+\eta\rmk
\rV
<\frac 1{4\sqrt2} \delta_{\rk}\lmk \frac \varepsilon 8\rmk e^{-\pi},\quad \text{and}\nonumber\\
\lV\lmk\unit-\hat \pi (\bar h)\rmk
\lmk\xi-\eta\rmk
\rV
<\frac 1{4\sqrt2} \delta_{\rk}\lmk \frac \varepsilon 8\rmk e^{-\pi},
\end{align}
%
for
\begin{align}\label{hbardef}
\bar h:=\sum_{x\in \caG_{\Lambda}} x\tilde hx^{*}.
\end{align}
Here the function $\delta_{\rk}$ is given in Theorem \ref{lem6kos}.
By this definition of $\bar h$, we see that 
\begin{align}
\bar h\in \lmk{\caB(\caH_{\alpha})}\otimes \caA_{\Lambda}\rmk'\cap
\left\{
\lambda_{g}\mid g\in G
\right\}'\cap
\lmk \caB(\caH_{\alpha})\otimes  C^{*}(\Sigma_{\Gamma}^{(\sigma)})\rmk_{+,1}.
\end{align}

We would like to replace $\bar h$ in (\ref{pmm}) to some positive element
$h \in (\caA_{\Gamma\setminus\Lambda}^{G})_{+,1}$.
In order to do so, we factorize $(\caH,\pi,u)$ to $\Lambda$-part and $\Gamma\setminus \Lambda$-part:
As in the proof of Theorem \ref{zbdip},
there exists an irreducible  covariant representation $(\tilde \caH,\tilde \pi,\tilde u)$ of 
$\Sigma_{\Gamma\setminus \Lambda}^{(\sigma)}$
and a unitary $W:\caH\to \lmk \bigotimes_{\Lambda}\bbC^{d}\rmk\otimes \tilde\caH$
such that
\begin{align}\label{prev1}
W\pi(a)W^{*}=\lmk \id_{\caA_{\Lambda}}\otimes \tilde\pi\rmk (a),\quad a\in\caA_{\Gamma},
\end{align}
and
\begin{align}\label{up1}
W u(g)W^{*}=\lmk \bigotimes_{\Lambda}U(g)\rmk\otimes \tilde u(g),\quad g\in G.
\end{align}
By Lemma \ref{pd}, $\tilde u$ has an irreducible
decomposition of given by
a set of Hilbert spaces 
$\{\caK_{\gamma}\mid \gamma \in  \pgs\}$.
By Lemma \ref{lem11iii} and Lemma \ref{pd} we have
\begin{align}\label{utd}
\tilde \pi\lmk\caA^{G}_{\Gamma\setminus \Lambda}\rmk''=\bigoplus_{\gamma\in \pgs}\unit_{\caH_{\gamma}}\otimes
B(\caK_{\gamma})=\tilde u(G)'.
\end{align}
Recall (\ref{pmm}).
Choose $\delta>0$ so that
\begin{align}\label{delt}
\lV
\hat\pi\lmk \bar h\rmk\lmk\xi+\eta\rmk
\rV+\delta
<\frac 1{4\sqrt2} \delta_{\rk}\lmk \frac \varepsilon 8\rmk e^{-\pi},\quad \text{and}\nonumber\\
\lV\lmk\unit-\hat \pi (\bar h)\rmk
\lmk\xi-\eta\rmk
\rV+\delta
<\frac 1{4\sqrt2} \delta_{\rk}\lmk \frac \varepsilon 8\rmk e^{-\pi}.
\end{align}
As $\bar h$ is in $\lmk{\caB(\caH_{\alpha})}\otimes \caA_{\Lambda}\rmk'$,
from  (\ref{prev1}), we see that
there exists a positive $y\in \caB(\tilde\caH)_{1}$ such that
\begin{align}
\lmk \unit_{\caH_{\alpha}}\otimes W\rmk
\hat\pi(\bar h) \lmk \unit_{\caH_{\alpha}}\otimes W^{*}\rmk
=\unit_{\caH_{\alpha}}\otimes\unit_{ \bigotimes_{\Lambda}\bbC^{d}}\otimes y.
\end{align}
Furthermore, as $\bar h$ is in $\left\{
\lambda_{g}\mid g\in G
\right\}'$, from (\ref{up1}),
$y$ belongs to $\tilde u(G)'=\tilde\pi (\caA^{G}_{\Gamma\setminus \Lambda})''$
by (\ref{utd}).
By the Kaplansky density theorem, there exists a positive $ h\in \lmk\caA^{G}_{\Gamma\setminus \Lambda}\rmk_{+,1}$
such that
\begin{align}
\lV
\lmk
\hat\pi(h)-\hat\pi(\bar h)
\rmk\lmk
\xi\pm\eta
\rmk
\rV
=
\lV
\lmk
\unit_{\caH_{\alpha}}\otimes\unit_{ \bigotimes_{\Lambda}\bbC^{d}}\otimes
\lmk \tilde\pi (h)-y\rmk
\rmk
\lmk 
\unit_{\caH_{\alpha}}\otimes W
\rmk
\lmk
\xi\pm\eta
\rmk
\rV<\delta.
\end{align}
This $h$ satisfies 
\begin{align}
&\lV
\hat\pi(h)\lmk
\xi+\eta
\rmk
\rV
\le 
\lV
\lmk \hat\pi(h)-\hat\pi(\bar h)\rmk\lmk
\xi+\eta
\rmk
\rV+\lV
\hat\pi(\bar h)\lmk
\xi+\eta
\rmk
\rV
< \frac 1{4\sqrt2} \delta_{\rk}\lmk \frac \varepsilon 8\rmk e^{-\pi},\quad \text{and}\nonumber\\
&\lV
\lmk \unit-\hat\pi(h)\rmk\lmk
\xi-\eta
\rmk
\rV
\le 
\lV
\lmk \hat\pi(h)-\hat\pi(\bar h)\rmk\lmk
\xi-\eta
\rmk
\rV+\lV
\lmk \unit-\hat\pi(\bar h)\rmk\lmk
\xi-\eta
\rmk
\rV
< \frac 1{4\sqrt2} \delta_{\rk}\lmk \frac \varepsilon 8\rmk e^{-\pi},
\end{align}
from the choice of $\delta$, (\ref{delt}).
We then obtain the required property of $h$:
\begin{align}
&\lV
e^{i\pi\hat\pi (h)}\xi-\eta
\rV
\le\lV
\frac 12 \lmk
e^{i\pi\hat\pi(h)}(\xi+\eta)-(\xi+\eta)\rmk\rV
+\lV \frac 12 \lmk e^{i\pi\hat\pi(h)}(\xi-\eta)+(\xi-\eta)
\rmk
\rV\nonumber\\
&\le
\frac{e^{\pi}}2\lV \hat\pi(h)\lmk\xi+\eta\rmk\rV
+\frac{e^{\pi}}2\lV \lmk \unit -\hat\pi(h)\rmk\lmk\xi-\eta\rmk\rV
<
\frac{1}{4\sqrt 2}\delta_{\rk}\lmk \frac\varepsilon 8\rmk.
\end{align}
\end{proof}
\begin{rem}
Note that an average over $G$ is contained in (\ref{hbardef}).
Because of this, we could take $\bar h$ to be $\Ad\lambda_{g}$-invariant.
This is possible because $\lambda_{g}$
is included in the $C^{*}$-algebra we consider, i.e., in $\caB(\caH_{\alpha})\otimes  C^{*}(\Sigma_{\Gamma}^{(\sigma)})$.

The main difference of Lemma \ref{lem4} compared to \cite{kos} is replacing
$\bar h$ with $h\in \caA_{\Gamma\setminus \Lambda}^{G}$.
To carry it out, the decomposition (\ref{utd})
given from Lemma \ref{pd} Lemma \ref{lem11iii}
is used. This decomposition reduces the problem to the Kaplansky density Theorem for 
$\hat\pi (\caA_{\Gamma\setminus\Lambda}^{G})$.
\end{rem}
\begin{notation}\label{lef}
For any $\varepsilon>0$ and a finite set
$\caF\Subset \caA$, there exists a $\Lambda(\varepsilon,\caF)\Subset\Gamma$
such that 
\begin{align}\label{ab}
\inf\left\{\lV a-b\rV\mid b\in \caA_{\Lambda(\varepsilon,\caF)}\right\}<\frac\varepsilon {16},\quad
\text{for all}\quad a\in\caF.
\end{align}
For each $\varepsilon>0$ and
$\caF\Subset \caA$, we fix such $\Lambda(\varepsilon,\caF)$.
If $\caF$ is included in $\caA_{\Lambda}$ for some $\Lambda\Subset \Gamma$,
we choose $\Lambda(\varepsilon,\caF)$ so that $\Lambda(\varepsilon,\caF)\subset \Lambda$.
For any $\varepsilon>0$ and 
$\caF\Subset \caA$,
set
\begin{align}
\delta_{\sjt}(\varepsilon,\caF):=\frac 12 \delta_{\nh}\lmk \frac \varepsilon 4, 
\Lambda(\varepsilon,\caF)\rmk.
\end{align}
Here we used 
the function $\delta_{\nh}$ introduced in Lemma \ref{lem4}.
\end{notation}

\begin{lem}\label{lem5}
Let $\varepsilon>0$, and 
$\caF\Subset (\caA_{\Gamma})_{1}$.
Let $(\caH,\pi,u)$ be
an irreducible  covariant representation of $\Sigma_{\Gamma}^{(\sigma)}$
with an irreducible decomposition of $u$ given by a set of Hilbert spaces 
$\{\caK_{\gamma}\mid \gamma \in  \pgs\}$.
Let 
$\xi,\eta$ be unit vectors in $\caK_{\alpha}$.
Suppose that unit vectors  
\begin{align}\label{txedef}
\tilde\xi:=\Omega_{\alpha}\otimes \xi,\quad 
\tilde\eta:=\Omega_{\alpha}\otimes\eta\in \caH_{\alpha}\otimes \caH
\end{align}
satisfy
\begin{align}\label{dare}
\lv
\braket{\tilde \eta}{\hat \pi(xy^{*}) \tilde\eta}
-\braket{\tilde \xi}{\hat \pi(xy^{*}) \tilde \xi}
\rv
<\delta_{\sjt}(\varepsilon,\caF),\quad
\text{for all}\quad x,y\in \caG_{\Lambda(\varepsilon,\caF)}.
\end{align}
(Recall Notation \ref{nagai} and Notation \ref{gca}.)
Then there exists a norm-continuous path of unitaries
$v:[0,1]\to\caU(\caA_\Gamma^{G})$ such that $v(0)=\unit_{\caA_{\Gamma}}$,
\begin{align}\label{tetx}
\tilde\eta= \lmk\unit_{\caH_{\alpha}}\otimes \pi(v(1))\rmk\tilde\xi,
\end{align}
and
\begin{align}\label{adv}
\sup_{t\in[0,1]}\lV
\Ad v(t)(a)-a
\rV<\varepsilon,\quad \text{for all}\quad a\in \caF.
\end{align}
\end{lem}
\begin{proof}
We denote by $\caN$, the finite dimensional subspace
spanned by $\{\hat \pi(xy^{*})\tilde\xi, \hat \pi(xy^{*})\tilde\eta\mid x,y\in \caG_{\Lambda(\varepsilon,\caF)}\}
$.
Then there exists a unit vector $\zeta$ in $\caN^{\perp}$, the orthogonal complement of $\caN$,
such that
\begin{align}\label{otcy}
\lv
\braket{\zeta}{\hat \pi(xy^{*}) \zeta}
-\braket{\tilde \xi}{\hat \pi(xy^{*}) \tilde \xi}
\rv<\delta_{\sjt}(\varepsilon,\caF)\le\delta_{\nh}\lmk \frac \varepsilon 4, 
\Lambda(\varepsilon,\caF)\rmk,\quad x,y\in \caG_{\Lambda(\varepsilon,\caF)}.
\end{align}
To see this, note that the intersection of the set of all compact operators on
$\caH_{\alpha}\otimes \caH$ and $\hat\pi\lmk \bcg\rmk$ is $0$ because
$\bcg$ is simple.
Applying Glimm's Lemma 
(Theorem \ref{glimm})
to $\delta_{\sjt}(\varepsilon,\caF)>0$,  a pure state $\braket{\tilde\xi}{\cdot \tilde\xi}$ on 
$\hat\pi\lmk \bcg\rmk$, 
a finite dimensional subspace $\caN$ of $\caH_{\alpha}\otimes \caH$
and 
a finite subset
$\hat\pi\lmk \caG_{\Lambda(\varepsilon,\caF)}\caG_{\Lambda(\varepsilon,\caF)}^{*}\rmk$,
we obtain $\zeta$ above.

Combining (\ref{otcy}) with (\ref{dare})we also get
\begin{align}
&\lv
\braket{\zeta}{\hat \pi(xy^{*}) \zeta}
-\braket{\tilde \eta}{\hat \pi(xy^{*}) \tilde \eta}
\rv\le
\lv
\braket{\zeta}{\hat \pi(xy^{*}) \zeta}
-\braket{\tilde \xi}{\hat \pi(xy^{*}) \tilde \xi}
\rv
+\lv
\braket{\tilde \eta}{\hat \pi(xy^{*}) \tilde\eta}
-\braket{\tilde \xi}{\hat \pi(xy^{*}) \tilde \xi}
\rv
\nonumber\\
&<2\delta_{\sjt}(\varepsilon,\caF)=\delta_{\nh}\lmk \frac \varepsilon 4,
\Lambda(\varepsilon,\caF)\rmk,\quad x,y\in \caG_{\Lambda(\varepsilon,\caF)}.
\end{align}

Hence
$(\caH,\pi,u)$
and unit vectors $\tilde\xi,\zeta$ (resp. $\tilde\eta, \zeta$) satisfy
{\it Condition 1.} (Notation  \ref{cond1}) for a pair $\delta_{\nh}\lmk \frac \varepsilon 4, 
\Lambda(\varepsilon,\caF)\rmk>0$, $\Lambda(\varepsilon,\caF)\Subset \Gamma$.
Therefore, from Lemma \ref{lem4}, 
there exist positive elements $h_{1},h_{2}$ in 
$(\caA_{\Gamma\setminus\Lambda(\varepsilon,\caF)}^{G})_{1}$
such that
\begin{align}\label{tiyu}
\lV
e^{i\pi\hat\pi (h_{1})}\tilde\xi-\zeta
\rV<
\frac{1}{4\sqrt 2}\delta_{\rk}\lmk \frac\varepsilon {32}\rmk,\quad\text{and}\quad
\lV
e^{i\pi\hat\pi (h_{2})}\tilde\eta-\zeta
\rV<
\frac{1}{4\sqrt 2}\delta_{\rk}\lmk \frac\varepsilon {32}\rmk.
\end{align}
Here $\delta_{\rk}$ is given in Theorem \ref{lem6kos}.
By the definition of $\tilde \xi$ (\ref{txedef}) and the decomposition
\begin{align}\label{ptpt}
\pi(a)=\bigoplus_{\gamma\in \pgs} \unit_{\caH_{\gamma}}\otimes \pi_{\gamma}(a),\quad
a\in \caA^{G}_{\Gamma}
\end{align}
((\ref{eq53}) of Lemma \ref{lem11iii}),
with irreducible $*$-representations $(\caK_{\gamma}, \pi_{\gamma})$
of $\caA_{\Gamma}^{G}$,
we have $e^{i\pi\hat\pi (h_{1})}\tilde\xi
=\Omega_{\alpha}\otimes e^{i\pi \pi_{\alpha}(h_{1)}}\xi$.
Similarly, we have
$e^{i\pi\hat\pi (h_{2})}\tilde\eta= \Omega_{\alpha}\otimes e^{i\pi \pi_{\alpha}(h_{2)}}\eta$.
Combining this with (\ref{tiyu}), we see that the unit vectors $e^{i\pi \pi_{\alpha}(h_{1)}}\xi, e^{i\pi \pi_{\alpha}(h_{2)}}\eta$ in $\caK_{\alpha}$
satisfies 
\begin{align}
\lV e^{i\pi \pi_{\alpha}(h_{1})}\xi-e^{i\pi \pi_{\alpha}(h_{2})}\eta\rV<\frac{1}{2\sqrt 2}\delta_{\rk}\lmk \frac\varepsilon {32}\rmk.
\end{align}
Then from Lemma \ref{lem9kos}, there exists a unitary $v_{0}$ on $\caK_{\alpha}$ such that
\begin{align}
v_{0}e^{i\pi \pi_{\alpha}(h_{1})}\xi=e^{i\pi \pi_{\alpha}(h_{2})}\eta,\quad\text{and}\quad
\lV
v_{0}-\unit_{\caK_{\alpha}}
\rV<\frac{1}{2}\delta_{\rk}\lmk \frac\varepsilon {32}\rmk<\delta_{\rk}\lmk \frac\varepsilon {32}\rmk.
\end{align}
From this and the fact that $\pi_{\alpha}$ is an irreducible 
representation of  $\caA^{G}_{\Gamma}$, 
applying Theorem \ref{lem6kos}, we obtain a self-adjoint
$k\in \caA^{G}_{\Gamma}$ such that
\begin{align}\label{piyo}
e^{i\pi_{\alpha}(k)}e^{i\pi \pi_{\alpha}(h_{1})}\xi=e^{i\pi \pi_{\alpha}(h_{2})}\eta,\quad\text{and}\quad
\lV k\rV\le \delta_{\nii}\lmk\frac{\varepsilon}{32}\rmk.
\end{align}
Here the function $\delta_{\nii}$ is given in Notation \ref{exp}.

Now we define a continuous path of unitaries
$v:[0,1]\to\caU(\caA_\Gamma^{G})$.
Set
\begin{align}
v_{1}(t):=e^{it\pi h_{1}}\in
 \caU\lmk\caA^{G}_{\Gamma\setminus \Lambda(\varepsilon,\caF)}\rmk,\quad
v_{2}(t):=e^{itk}\in\caU\lmk\caA^{G}_{\Gamma}\rmk,\quad
v_{3}(t):=e^{-it\pi h_{2}}\in
 \caU\lmk\caA^{G}_{\Gamma\setminus \Lambda(\varepsilon,\caF)}\rmk
\end{align}
for each $t\in[0,1]$.

For $i=1,3$, as $v_{i}$ takes value in $\caU\lmk\caA^{G}_{\Gamma\setminus \Lambda(\varepsilon,\caF)}\rmk$, $v_{i}(t)$ commutes with elements in $ \caA_{\Lambda(\varepsilon,\caF)}$. From this and the fact that
the distance between $\caF$ and $ \caA_{\Lambda(\varepsilon,\caF)}$  is
 less than $\frac{\varepsilon}{16}$ (Notation \ref{lef} (\ref{ab})),
 we get $\lV \Ad v_{i}(t) (a)-a\rV<\frac\varepsilon 8$, for all $a\in\caF$,
 $t\in[0,1]$, and $i=1,3$.
For $i=2$, from $\lV k\rV\le\delta_{\nii}\lmk\frac{\varepsilon}{32}\rmk$, recalling the definition of
$\delta_{\nii}$ in Notation \ref{exp}, 
we obtain $\lV \Ad v_{2}(t)(a)-a\rV\le 2\lV  v_{2}(t)-\unit \rV\le\frac \varepsilon{16}$, for all $a\in\caF\subset\lmk \caA_{\Gamma}\rmk_{1}$ and $t\in[0,1]$.

We define $v:[0,1]\to\caU\lmk \caA^{G}_{\Gamma}\rmk$ by
\begin{align}
v(t):=\left\{
\begin{gathered}
v_{1}(3t),\quad \quad t\in\left[0, \frac 13\right],\\
v_{2}\lmk 3\lmk t-\frac 13\rmk\rmk v_{1}(1),\quad\quad t\in\left[\frac 13,\frac 23\right],\\
v_{3}\lmk 3\lmk t-\frac 23\rmk\rmk v_{2}(1) v_{1}(1),\quad\quad t\in\left[\frac 23,1\right].
\end{gathered}
\right.
\end{align}
Clearly $v(0)=\unit_{\caA_{\Gamma}}$ and $v$ is norm-continuous, and it takes values in 
$\caU\lmk \caA^{G}_{\Gamma}\rmk$.
 From the above estimates on
$\lV \Ad v_{i}(t)(a)-a\rV$ for $a\in\caF$ and $i=1,2,3$, we also get (\ref{adv}).
Furthermore, we have
\begin{align}
 \lmk\unit_{\caH_{\alpha}}\otimes \pi(v(1))\rmk\tilde\xi
 =\Omega_{\alpha}\otimes \pi_{\alpha}\lmk v(1)\rmk\xi
=\Omega_{\alpha}\otimes \pi_{\alpha}\lmk e^{-i\pi h_{2}}e^{ik}e^{i\pi h_{1}}\rmk\xi
=\Omega_{\alpha}\otimes\eta=\tilde\eta.
\end{align}
Here, for the first equality, we used the fact that $v(1) $ is in  in $\caA^{G}_{\Gamma}$ and
(\ref{ptpt}).
The third equality is from (\ref{piyo}).
\end{proof}
\begin{rem}
By Lemma \ref{lem4}, we can take $h_{1}, h_{2}$ in 
the fixed point algebra $\caA^{G}_{\Gamma\setminus \Lambda_{(\varepsilon,\caF)}}$.
With the special form of $\tilde\xi,\tilde\eta$, in (\ref{txedef}),
the problem is reduced to the Kadison transitivity theorem
for $(\caK_{\alpha},\pi_{\alpha}(\caA_{\Gamma}^{G}))$.
The irreducibility of $\pi_{\alpha}$ is used there.
From this we may obtain $k$ interpolating
$e^{i\pi\hat\pi (h_{1})}\tilde\xi$ and $e^{i\pi\hat\pi (h_{2})}\tilde\eta$,
 from $\caA^{G}_{\Gamma}$.

\end{rem}

\begin{lem}\label{lem8}
For any $\varepsilon>0$ and $\caF\Subset\lmk \caA_{\Gamma}\rmk_{1}$, the following holds:
Let $(\caH_{i},\pi_{i},u_{i})$ with $i=0,1$ be
irreducible  covariant representations of $\Sigma_{\Gamma}^{(\sigma)}$
with irreducible decomposition of $u_{i}$ given by a set of Hilbert spaces 
$\{\caK_{\gamma,i}\mid \gamma \in  \pgs\}$.
Let 
$\xi_{i}\in\caK_{\alpha,i}$ be a unit vector in $\caK_{\alpha,i}$ for $i=0,1$.
Suppose that the representation $\hat\pi_{i}:=\id_{\caH_{\alpha}}\otimes \lmk \pi_{i}\times u_{i}\rmk$, $i=0,1$
of $B(\caH_{\alpha})\otimes C^{*}(\Sigma_{\Gamma}^{(\sigma)})$ 
and unit vectors $\tilde\xi_{i}:=\Omega_{\alpha}\otimes \xi_{i}$ 
in $\caH_{\alpha}\otimes \caH_{i}$, $i=0,1$ satisfy 
\begin{align}\label{toto}
\lv
\braket{\tilde \xi_{0}}{\hat\pi_{0}(x y^{*})\tilde\xi_{0}}-\braket{\tilde \xi_{1}}{\hat\pi_{1}(x y^{*})\tilde\xi_{1}}
\rv<\frac 12\delta_{\sjt}\lmk \varepsilon,\caF\rmk,\quad
\text{for all }\quad  x,y\in \caG_{\Lambda(\varepsilon,\caF)}.
\end{align}
(Recall Notation \ref{lef} for $\delta_{\sjt}$.)
Then for any $\varepsilon'>0$ and $\caF'\Subset B(\caH_{\alpha})\otimes C^{*}(\Sigma_{\Gamma}^{(\sigma)})$,
there exists a norm-continuous path
$v:[0,1]\to \caU(\caA_{\Gamma}^{G})$ with $v(0)=\unit_{\caA_{\Gamma}}$ such that 
\begin{align}\label{zen}
\lv
\braket{\tilde \xi_{0}}{\hat\pi_{0}(a)\tilde\xi_{0}}
-\braket{\tilde \xi_{1}}{\lmk \hat\pi_{1}\circ \Ad\lmk v(1)\rmk\rmk(a)\tilde\xi_{1}}
\rv<\varepsilon',\quad
\text{for all }\quad a\in\caF',
\end{align}
and 
\begin{align}
\lV
\Ad v(t)(y)-y
\rV<\varepsilon,\quad
\text{for all }\quad y\in \caF,\quad
\text{and }\quad t\in[0,1].
 \end{align}
\end{lem}

\begin{proof}
From Lemma \ref{lem2},  there exists a self-adjoint $h\in \caA_{\Gamma}^{G}$
such that
\begin{align}\label{lem2mains}
\lv
\braket{\tilde \xi_{0}}{\hat\pi_{0}(a)\tilde\xi_{0}}-\braket{\tilde \xi_{1}}{\hat\pi_{1}\circ\Ad \lmk e^{ih}\rmk(a)\tilde\xi_{1}}
\rv<\min\left\{\varepsilon', \frac 12 \delta_{\sjt}\lmk \varepsilon,\caF\rmk\right\}\quad \text{for all}\quad a\in \caF'\cup \caG_{\Lambda\lmk \varepsilon,\caF\rmk}\lmk\caG_{\Lambda\lmk \varepsilon,\caF\rmk}\rmk^{*}.
\end{align}
From this and (\ref{toto}), we have
\begin{align}\label{108y}
\lv
\braket{\tilde \xi_{1}}{\hat\pi_{1}\circ\Ad \lmk e^{ih}\rmk(xy^{*})\tilde\xi_{1}}
-\braket{\tilde \xi_{1}}{\hat\pi_{1}(xy^{*})\tilde\xi_{1}}
\rv<
\delta_{\sjt}\lmk \varepsilon,\caF\rmk\quad
\text{for all}\quad x,y\in \caG_{\Lambda\lmk \varepsilon,\caF\rmk}.
\end{align}
Recall from Lemma \ref{lem11iii} that 
\begin{align}\label{ptpt1}
\pi_{1}(a)=\bigoplus_{\gamma\in \pgs} \unit_{\caH_{\gamma}}\otimes \pi_{\gamma,1}(a),\quad
a\in \caA^{G}_{\Gamma}
\end{align}
with irreducible $*$-representations $(\caK_{\gamma,1}, \pi_{\gamma,1})$
of $\caA_{\Gamma}^{G}$.
From this and $h\in\caA_{\Gamma}^{G}$, we see that
$\hat\pi_{1}\lmk e^{-ih}\rmk\tilde\xi_{1}=\Omega_{\alpha}\otimes  \pi_{\alpha, 1}(e^{-ih})\xi_{1}$
By (\ref{108y}),
$(\caH_{1},\pi_{1},u_{1})$,
$\xi_{1}$, $\pi_{\alpha, 1}(e^{-ih})\xi_{1}$
satisfies the required condition in Lemma \ref{lem5}.

Applying Lemma \ref{lem5} for $(\caH_{1},\pi_{1},u_{1})$ and $ \xi_{1}$,
$\pi_{\alpha,1}(e^{-ih})\xi_{1}$,
 we obtain a norm-continuous path of unitaries
$v:[0,1]\to\caU(\caA_\Gamma^{G})$ such that $v(0)=\unit_{\caA_{\Gamma}}$,
\begin{align}\label{tetx}
\hat\pi_{1}\lmk e^{-ith} \rmk \tilde\xi_{1}=\hat\pi_{1} \lmk v(1)^{*}\rmk\tilde\xi_{1},
\end{align}
and
\begin{align}\label{adv1}
\sup_{t\in[0,1]}\lV
\Ad v(t)(a)-a
\rV<\varepsilon,\quad \text{for all}\quad a\in \caF.
\end{align}
From (\ref{lem2mains}) and (\ref{tetx}), we obtain (\ref{zen}).
\end{proof}
\begin{rem}
As in \cite{kos}, we replace $e^{ith}$ with $v(t)$ which satisfy (\ref{adv1}).
We may do so with $v(t)$ in $\caA_{\Gamma}^{G}$ because of Lemma \ref{lem5}.

\end{rem}

After these preparation, the proof of Proposition \ref{lem6}
is the same as proof of Theorem 2.1 of \cite{kos}.
We give it here for the reader's convenience.
\begin{proofof}[Proposition \ref{lem6}]
We fix an increasing sequence $\Lambda_{n}$, $n=0,1,2,\ldots$
of non-empty finite subsets of $\Gamma$ such that $\Lambda_{n}\nearrow \Gamma$.

For each $i=0,1$, we use the notation $\hat\pi_{i}$, $\tilde\xi_{i}$, $\hat\varphii_{{\xi}_{i}}$ given in 
Notation \ref{nagai}, replacing $(\caH,\pi,u)$ and $\xi\in\caK_{\alpha}$ 
with $(\caH_{i},\pi_{i},u_{i})$ and $\xi_{i}\in\caK_{\alpha,i}$.
Let $(\caK_{\alpha,i}, \pi_{\alpha,i})$, $i=0,1$ be the irreducible $*$-representation of $\caA_{\Gamma}^{G}$ obtained in Lemma \ref{lem11iii} (\ref{eq53})
with $(\caH,\pi,u)$, $\Gamma_{0}$ replaced by $(\caH_{i},\pi_{i},u_{i})$, $\Gamma$.

Set $\caF_{0}:=\caS_{\Lambda_{0}}$. (Recall (\ref{sldef}).)
Fix $\varepsilon>0$ or set $\varepsilon=1$.
Set  $\caG_{0}:=
\caG_{\Lambda\lmk{\varepsilon},\caF_{0}\rmk}\caG_{\Lambda\lmk{\varepsilon},\caF_0\rmk}^{*}$.
From Lemma \ref{lem2}, there exists a self-adjoint $h_{0}\in\caA_\Gamma^{G}$.
such that
\begin{align}\label{kaze}
\lv
\hat\varphii_{{\xi}_{0}}\circ \Ad\lmk e^{ih_{0}}\rmk\lmk a\rmk
-\hat\varphii_{{\xi}_{1}}(a)
\rv<\min\left\{\frac12 \delta_{{\sjt}}\lmk
{\varepsilon},\caF_{0}
\rmk, {\varepsilon}\right\},\quad
a\in  \caG_{0}\cup\caF_{0}.
\end{align}
(Recall Notation \ref{lef} for $\delta_{\sjt}$.)
We define $v_{0}:=[0,1]\to \caU(\caA_\Gamma^{G})$ by 
\begin{align}\label{v0def}
v_{0}(t)=e^{ith_{0}},\quad t\in[0,1].
\end{align}
 We consider the following proposition $[P_{n}]$ for
 each $n\in \nan$:
 \begin{quote} {[$P_{n}$]}
 There exist norm-continuous paths $v_{k}:[0,1]\to \caU(\caA_\Gamma^{G})$, $k=0,\ldots, 2n$
 with $v_{k}(0)=\unit_{\caA_{\Gamma}}$
 satisfying the following:
 Set
 \begin{align}\label{fe}
 \caF_{2j}:= \left\{
 x,
 \Ad\lmk
 v_{2j-1}(1)^{*} v_{2j-3}(1)^{*}\cdots v_{3}(1)^{* }v_{1}(1)^{*}
 \rmk(x)\mid x\in \caS_{\Lambda_{2j}}
 \right\}, \quad j=1,\ldots, n,
 \end{align}
 and
 \begin{align}\label{fo}
 \caF_{2j-1}:= \left\{
 x,
 \Ad\lmk
 v_{2j-2}(1)^{*} v_{2j-4}(1)^{*}\cdots v_{4}(1)^{* }v_{2}(1)^{*}v_{0}(1)^{*}
 \rmk(x)\mid x\in \caS_{\Lambda_{2j-1}}
 \right\}, \quad j=1,\ldots, n.
 \end{align}
 (Recall (\ref{sldef}).)
 We also denote the finite subset $\caG_{\Lambda\lmk\frac{\varepsilon}{2^{k}},\caF_{k}\rmk}\caG_{\Lambda\lmk\frac{\varepsilon}{2^{k}},\caF_{k}\rmk}^{*}$
by $\caG_{k}$, for each $k=0,1,\ldots,2n$.
Then
the following three
inequalities hold.
\begin{enumerate}
\item
For all $a\in \caG_{2n}\cup\caF_{2n}$,
\begin{align}\label{kono}
\lv
\hat\varphii_{{\xi}_{0}}\circ \Ad\lmk
v_{0}(1)v_{2}(1)\cdots v_{2n}(1)
\rmk(a)
-\hat\varphii_{{\xi}_{1}}\circ \Ad\lmk
v_{1}(1)v_{3}(1)\cdots v_{2n-1}(1)
\rmk(a)
\rv
<\min\left\{\frac12 \delta_{{\sjt}}\lmk
\frac{\varepsilon}{2^{2n}},\caF_{2n}
\rmk, \frac{\varepsilon}{2^{2n}}\right\}
 \end{align}
\item 
For all $a\in \caG_{2n-1}\cup\caF_{2n-1}$,
\begin{align}\label{hito}
\lv
\hat\varphii_{{\xi}_{0}}\circ \Ad\lmk
v_{0}(1)v_{2}(1)\cdots v_{2n-2}(1)
\rmk(a)
-\hat\varphii_{{\xi}_{1}}\circ \Ad\lmk
v_{1}(1)v_{3}(1)\cdots v_{2n-1}(1)
\rmk(a)
\rv
<\min\left\{\frac12 \delta_{{\sjt}}\lmk
\frac{\varepsilon}{2^{2n-1}},\caF_{2n-1}
\rmk, \frac{\varepsilon}{2^{2n-1}}\right\}
\end{align}
\item For all $t\in[0,1]$, $k=1,2,\ldots,2n$ with $k\le 2n$ and $x\in \caF_{k-1}$,
we have
\begin{align}\label{vapd}
\lV
\Ad v_{k}(t)(x)-x
\rV<\frac\varepsilon{2^{k-1}}.
\end{align}
\end{enumerate}
%
%
%
%
 \end{quote}
Let us check that [$P_{1}$] with $v_{0}$ given in (\ref{v0def}) holds. 
Set $\caF_{1}$ as in (\ref{fo}) with $j=1$ and this $v_{0}$.
Set $\caG_{1}:=\caG_{\Lambda\lmk\frac{\varepsilon}{2},\caF_{1}\rmk}\caG_{\Lambda\lmk\frac{\varepsilon}{2},\caF_{1}\rmk}^{*}$.
From (\ref{kaze}), applying 
Lemma \ref{lem8} for vectors $\xi_{1}$ and $\pi_{\alpha,0}(v_{0}(1)^{*})\xi_{0}$,
there exists a norm-continuous path
$v_{1}:[0,1]\to \caU\lmk\caA_{\Gamma}^{G}\rmk$ with $v_{1}(0)=\unit$
such that
\begin{align}\label{p1}
\lv
\hat\varphii_{{\xi}_{0}}\circ\Ad \lmk v_{0}(1)\rmk\lmk a \rmk
-\hat\varphii_{{\xi}_{1}}\circ\Ad \lmk v_{1}(1)\rmk\lmk a \rmk
\rv
<\min\left\{\frac12 \delta_{{\sjt}}\lmk
\frac{\varepsilon}{2},\caF_{1}\rmk,
\frac{\varepsilon}{2}\right\}
,\quad
\text{for all }\quad a\in \caG_{1}\cup\caF_{1}
\end{align}
and 
\begin{align}
\lV
\Ad v_{1}(t)(y)-y
\rV<\varepsilon,\quad
\text{for all }\quad y\in \caF_{0},\quad
\text{and }\quad t\in[0,1].
 \end{align}
 Set $\caF_{2}$ as in (\ref{fe}) with $j=1$ for this $v_{1}$.
 And set $\caG_{2}:=\caG_{\Lambda\lmk\frac{\varepsilon}{2^{2}},\caF_{2}\rmk}\caG_{\Lambda\lmk\frac{\varepsilon}{2^{2}},\caF_{2}\rmk}^{*}$.
From (\ref{p1}), applying Lemma \ref{lem8} again to
vectors $\pi_{\alpha,0}\lmk v_{0}(1)^{*}\rmk\xi_{0}$ and
$\pi_{\alpha,1}\lmk v_{1}(1)^{*}\rmk\xi_{1}$, we obtain
 a norm-continuous path
$v_{2}:[0,1]\to \caU\lmk\caA_{\Gamma}^{G}\rmk$ with $v_{2}(0)=\unit$
such that
\begin{align}
\lv
\hat\varphii_{{\xi}_{0}}\circ\Ad \lmk v_{0}(1)v_{2}(1)\rmk\lmk a \rmk
-\hat\varphii_{{\xi}_{1}}\circ\Ad \lmk v_{1}(1)\rmk\lmk a \rmk
\rv
<\min\left\{\frac12 \delta_{{\sjt}}\lmk
\frac{\varepsilon}{2^{2}},\caF_{2}\rmk,
\frac{\varepsilon}{2^{2}}\right\}
,\quad
\text{for all }\quad a\in \caG_{2}\cup\caF_{2}
\end{align}
and 
\begin{align}
\lV
\Ad v_{2}(t)(y)-y
\rV<\frac \varepsilon 2,\quad
\text{for all }\quad y\in \caF_{1},\quad
\text{and }\quad t\in[0,1].
 \end{align}
Hence we have proven [$P_{1}$] with $v_{0}$ given in (\ref{v0def}).
The proof that [$P_{n}$] implies [$P_{n+1}$] with the same $v_{0},v_{1},\ldots, v_{2n}$
as in  [$P_{n}$]
can be carried out in the same way, by the repeated use of Lemma \ref{lem8}.
Hence we obtain a sequence $\{v_{n}\}_{n=0}^{\infty}$ of norm-continuous paths
 $v_{n}:[0,1]\to \caU(\caA_{\Gamma}^{G})$ with $v_{n}(0)=\unit_{\caA_{\Gamma}}$
 satisfying (\ref{kono}) (\ref{hito}) (\ref{vapd}).

We define norm continuous paths $y,z:[0,\infty)\to \caU(\caA_{\Gamma}^{G})$
by
\begin{align}
&y(t):=v_{1}(t) v_{3}(t)\cdots v_{2j-1}(1)v_{2j+1}(t-[t]),\quad j\le t<j+1,\quad
j=0,1,2,\ldots,\nonumber\\
&z(t):=v_{0}(t) v_{2}(t)\cdots v_{2j-2}(1)v_{2j}(t-[t]),\quad j\le t<j+1\quad
j=0,1,2,\ldots.
\end{align}
Here $[t]$ denotes the largest integer less than or equal to $t$.
Then as in section 2 of \cite{kos}, 
for any $a\in\caA_{\rm loc,\Gamma}$,
the limit
\begin{align}
\gamma_{0}(a):=\lim_{t\to\infty}\Ad\lmk z(t)\rmk(a),\quad
\gamma_{1}(a):=\lim_{t\to\infty}\Ad\lmk y(t)\rmk(a)
\end{align}
exist because of (\ref{vapd}) and the fact that $\caS_{\Lambda_{n}}\subset \caF_{n}$.
These limit define endomorphisms $\gamma_{0},\gamma_{1}$ on $\caA_{\Gamma}$.
Furthermore, because of (\ref{vapd}) and the fact that
$\Ad\lmk v_{n-1}(1)^{*}v_{n-3}(1)^{*}\cdots\rmk\lmk \caS_{\Lambda_{n}}\rmk\subset \caF_{n}$,
by the definition (\ref{fe}) and (\ref{fo}), for any $x\in\caA_{\rm loc,\Gamma}$,
the limit
\begin{align}
&\lim_{j\to\infty} \Ad\lmk v_{2j}(1)^{* } v_{2j-2}(1)^{*}\cdots v_{0}(1)^{*}\rmk(x)
=:a_{x}  \\
&\lim_{j\to\infty} \Ad\lmk v_{2j-1}(1)^{* } v_{2j-3}(1)^{*}\cdots v_{1}(1)^{*}\rmk(x)
=:b_{x}
\end{align}
exist. For these limits, we have
$\gamma_{0}(a_{x})=x$, and $\gamma_{1}(b_{x})=x$,  for all $x\in\caA_{\rm loc,\Gamma}$.
Therefore, $\gamma_{0}$ and $ \gamma_{1}$ are automorphisms.
By {\it 1.}, {\it 2.} of [$P_{n}$], we also have
\begin{align}\label{l1l2}
\varphii_{0}\circ\gamma_{0}=
\left.\hat\varphii_{{\xi}_{0}}\right\vert_{\caA_{\Gamma}}\circ\gamma_{0}
=\left.\hat\varphii_{{\xi}_{1}}\right\vert_{\caA_{\Gamma}}\circ\gamma_{1}
=\varphii_{1}\circ\gamma_{1}.
\end{align}
Let $\Xi_{\Gamma}$ be an automorphism given by $\Xi_{\Gamma}:=\gamma_{0}\circ\gamma_{1}^{-1}$ on $\caA$.
Define a norm-continuous path $w:[0,\infty)\to \caU(\caA_{\Gamma}^{G})$
by
\begin{align}
w(t):=z(t)y(t)^{*},\quad t\in [0,\infty).
\end{align}
We have
\begin{align}
\Xi_{\Gamma}(x)=\gamma_{0}\circ\gamma_{1}^{-1}(x)
=\lim_{t\to\infty} \Ad(w(t))(x),\quad x\in\caA_{\Gamma},
\end{align}
and $w(0)=\unit$.
From (\ref{l1l2}), we have $\varphii_{0}\circ\Xi_{\Gamma}=\varphii_{1}$.
This completes the proof.

\end{proofof}

\section{Proof of the Main Theorem}
Now we are ready to prove Theorem \ref{main}.
Let $\omega_{0}$ and $\omega_{1}$ be elements of $SPG(\caA)$. 
\begin{proofof}["if'' part of Theorem \ref{main}]
Suppose that $c_{\omega_{0},R}=c_{\omega_{1},R}$.
From Lemma \ref{clmcr}, we have $c_{\omega_{0},L}=c_{\omega_{1},L}$.
For each $\zeta=L,R$ and $i=0,1$, let  
$(\caL_{\omega_{i},{\vp}}, \rho_{\omega_{i},{\vp}}, u_{\omega_{i},{\vp}},\sigma_{\omega_{i},{\vp}})$
be a quadruple associated to $(\omega_{i}\vert_{\caA_{{\vp}}},\tau_{\vp})$.
By Remark \ref{rem17},
we may assume that $\sigma_{R}:=\sigma_{\omega_{0},{R}}=\sigma_{\omega_{1},{R}}$
and $\sigma_{L}:=\sigma_{\omega_{0},{L}}=\sigma_{\omega_{1},{L}}$.
For each $\zeta=L,R$ and $i=0,1$,
the triple $(\caL_{\omega_{i},{\vp}},  \rho_{\omega_{i},{\vp}}, u_{\omega_{i},{\vp}})$
is an irreducible  covariant representations of 
the twisted $C^{*}$-dynamical system $\Sigma_{\Gamma_{\vp}}^{(\sigma_{\vp})}$.
By Lemma \ref{pd}, $u_{\omega_{i},{\vp}}$ has an 
 irreducible decomposition given by a set of Hilbert spaces 
$\{\caK_{\gamma,i,\vp}\mid \gamma \in  \caP_{\sigma_{\vp}}\}$.
For each $\vp=L,R$, fix some $\alpha_{\vp}\in  \caP_{\sigma_{\vp}}$.
The spaces $\caK_{\alpha_{\vp},i,\vp}$ $i=0,1$ are non-zero
because of Proposition \ref{zbdi}.
Fix unit vectors $\xi_{i,\vp}\in \caK_{\alpha_{\vp},i,\vp}$  for each $\vp=L,R$ and $i=0,1$.

For each  $\vp=L,R$ and $i=0,1$,
let $\hat\varphi_{\xi_{i,\vp}}$ be a state on $\caB(\caH_{\alpha_{\vp}})\otimes C^{*}(\Sigma_{\vp}^{{(\sigma_{\vp})}})$
given by $(\caL_{\omega_{i},{\vp}}, \rho_{\omega_{i},{\vp}}, u_{\omega_{i},{\vp}},\sigma_{\omega_{i},{\vp}})$
(defined in Notation \ref{nagai} (\ref{nekosan}) with
$\caH,\pi,u,\xi$ replaced by $\caL_{\omega_{i},{\vp}}, \rho_{\omega_{i},{\vp}}, u_{\omega_{i},{\vp}},\xi_{i,\vp}$).
Let
$\varphi_{i,\vp}$ be the restriction of $\hat\varphi_{\xi_{i,\vp}}$ 
onto $\caA_{\Gamma_{\vp}}$.
By the definition, $\varphi_{0,\vp}$, $\varphi_{1,\vp}$ are 
quasi-equivalent to $\omega_{0}\vert_{\caA_{\vp}}$, $\omega_{1}\vert_{\caA_{\vp}}$, 
respectively.

By Proposition \ref{lem6}, there exist $\Xi_{\vp}\in \ainn^{G}(\caA_{\vp})$
such that $\varphi_{1,\vp}=\varphi_{0,\vp }\circ \Xi_{\vp}$,
$\vp=L,R$.
Recall that $\omega_{0}$, $\omega_{1}$ are quasi-equivalent to 
$\omega_{0}\vert_{\caA_{L}}\otimes\omega_{0}\vert_{\caA_{R}}$ and
$\omega_{1}\vert_{\caA_{L}}\otimes\omega_{1}\vert_{\caA_{R}}$ respectively from the
split property. (Remark \ref{splitrem}.)
Hence we obtain
\begin{align}
&\omega_{1}\sim_{\rm q.e.}\omega_{1}\vert_{\caA_{L}}\otimes\omega_{1}\vert_{\caA_{R}}
\sim_{\rm q.e.}
\varphi_{1,L}\otimes\varphi_{1,R}
=
\lmk \varphi_{0,L}\circ \Xi_{L}\rmk\otimes\lmk\varphi_{0,R}\circ \Xi_{R}\rmk
=\lmk
\varphi_{0,L}\otimes\varphi_{0,R}\rmk
\circ \lmk \Xi_{L}\otimes \Xi_{R}\rmk\nonumber\\
&\sim_{\rm q.e.}
\lmk
\omega_{0}\vert_{\caA_{L}}\otimes\omega_{0}\vert_{\caA_{R}}
\rmk
\circ \lmk \Xi_{L}\otimes \Xi_{R}\rmk
\sim_{\rm q.e.}
\omega_{0}
\circ \lmk \Xi_{L}\otimes \Xi_{R}\rmk.
\end{align}
This completes the proof.
\end{proofof}

\begin{proofof}["only if'' part of Theorem \ref{main}]
Suppose 
that $\omega_{0}\sim_{{\rm split},\tau}\omega_{1}$.
Then
there exist automorphisms $\Xi_{L}\in \ainn^{G}(\caA_{L})$ 
and $\Xi_{R}\in \ainn^{G}(\caA_{R})$ such that
$\omega_{1}$ and $\omega_{0}\circ\lmk\Xi_{L}\otimes \Xi_{R}\rmk$
are quasi-equivalent.
From the split property, we have
$\omega_{1}\sim_{\rm q.e.}\omega_{1}\vert_{\caA_{L}}\otimes\omega_{1}\vert_{\caA_{R}}$ and $\omega_{0}\circ\lmk\Xi_{L}\otimes \Xi_{R}\rmk\sim_{\rm q.e.}
\lmk
\omega_{0}\vert_{\caA_{L}}\circ\Xi_{L}
\rmk\otimes
\lmk
\omega_{0}\vert_{\caA_{R}}\circ\Xi_{R}
\rmk$.
Combining these, we see that
$\omega_{1}\vert_{\caA_{R}}$ and $\omega_{0}\vert_{\caA_{R}}\circ\Xi_{R}
$
are quasi-equivalent.

For each $\zeta=L,R$ and $i=0,1$, let  
$(\caL_{\omega_{i},{\vp}}, \rho_{\omega_{i},{\vp}}, u_{\omega_{i},{\vp}},\sigma_{\omega_{i},{\vp}})$
be a quadruple associated to $(\omega_{i}\vert_{\caA_{{\vp}}},\tau_{\vp})$.
From $\omega_{1}\vert_{\caA_{R}}\sim_{\rm q.e.}\omega_{0}\vert_{\caA_{R}}\circ\Xi_{R}$,
$\rho_{\omega_0,{R}}\circ\Xi_{R}$ is an irreducible $*$-representation of
$\caA_{R}$  on $\caL_{\omega_{0},{R}}$, which is quasi-equivalent
to the GNS representation of $\omega_{1}\vert_{\caA_{R}}$.
Furthermore, the projective unitary representation 
$u_{\omega_{0},{R}}$ of $G$ with
$2$-cocycle $\sigma_{\omega_{0},{R}}$
satisfies
\begin{align}
\rho_{\omega_{0},{R}}\circ\Xi_{R}\circ \tau_{R}(g)\lmk a\rmk
=\rho_{\omega_{0},{R}}\circ \tau_{R}(g) \circ\Xi_{R}\lmk a\rmk
=\Ad\lmk u_{\omega_{0},{R}}(g)\rmk\circ \rho_{\omega_{0},{R}}\circ\Xi_{R}\lmk a\rmk
,\quad a\in\caA_{R},\quad g\in G.
\end{align}
From this,
$(\caL_{\omega_{0},{R}}, \rho_{\omega_{0},{R}}\circ\Xi_{R}, 
u_{\omega_{0},{R}},\sigma_{\omega_{0},{R}})$
is a quadruple associated to $(\omega_{1}\vert_{\caA_{{R}}},\tau_{R})$.
Hence we obtain $c_{\omega_{1},R}=c_{\omega_{0},R}$.
  This proves the claim.
\end{proofof}

{\bf Acknowledgment.}\\
{
This work was supported by JSPS KAKENHI Grant Number 16K05171 and 19K03534.
}

\appendix

\section{Basic Notation}\label{notasec}
For a Hilbert space $\caH$, $B(\caH)$ denotes the set of all bounded operators on $\caH$.
If $V:\caH_1\to\caH_2$ is a linear/anti-linear map from a Hilbert space $\caH_1$ to 
another Hilbert space $\caH_2$,
then $\Ad (V):B(\caH_1)\to B(\caH_2)$ denotes the map
$\Ad(V)(x):=V x V^*$, $x\in B(\caH_1)$.

For a set $\caX$, $\caF\Subset \caX$ means that $\caF$ is a finite subset of $\caX$. For a finite
set $S$, $|S|$ indicates the number of elements in $S$.

For a $C^{*}$-algebra $\caB$, we denote by
$\caB_{1}$ the set of all elements
in  $\caB$ with norm less than or equal to $1$ and by $\caB_{+,1}$ the set of all positive elements
in  $\caB_{1}$.
For a state $\omega$, $\varphi$ on a $C^{*}$-algebra $\caB$,
we write $\omega\sim_{\rm q.e.}\varphi$ when they are quasi-equivalent.
We denote by $\Aut \caB$ the group of automorphisms on a $C^{*}$-algebra $\caB$.
For a unital $C^{*}$-algebra $\caB$, the unit of $\caB$ is denoted by $\unit_{\caB}$.
For a Hilbert space we write $\unit_{\caH}=\unit_{\caB(\caH)}$.
For a unital $C^{*}$-algebra $\caB$, by $\caU(\caB)$, we mean
the set of all unitary elements in $\caB$.
For a Hilbert space we write $\caU(\caH)$ for $\caU(\caB(\caH))$.
For a $C^{*}$-algebra $\caB$ and $v\in \caB$, we set 
$\Ad(v)(x):=vxv^{*}$, $x\in \caB$.
For a state $\varphi$ on $\caB$ and a $C^{*}$-subalgebra $\caC$ of $\caB$,
$\varphi\vert_{\caC}$ indicates the restriction of $\varphi$ to $\caC$.

\section{Facts from \cite{kos} and \cite{fkk}}\label{kosfkk}
In this section, we list up facts used/proven in \cite{kos} and \cite{fkk}.
\begin{lem}\label{kos1312}
Let $\varphii_{i}$, $i=0,1$ be pure states on a simple unital $C^{*}$-algebra
$\caA$ with GNS triple $(\caH_{i},\pi_{i},\Omega_{i})$.
Then for all $\caF\Subset \caA$ and $\varepsilon>0$
there exists 
an $f\in \caA_{+,1}$ and a unit vector $\zeta\in\caH_{1}$ such that
\begin{align}\pi_{1}(f)\zeta=\zeta,\quad
\lV
f\lmk a-\varphii_{0}(a)\unit_{\caA} \rmk f
\rV<\varepsilon,\quad \text{for all}\quad a\in \caF.
\end{align}
\end{lem}
\begin{proof}
See proof of Lemma 2.3 of \cite{fkk}.
\end{proof}
By  a basic consideration of $2$-dimensional Hilbert space, we obtain the following.
\begin{lem}\label{lem9kos}
For any $\varepsilon>0$, a Hilbert space $\caH$, and
unit vectors $\xi_{1},\xi_{2}\in\caH$ with $\lV \xi_{1}-\xi_{2}\rV<\frac1{\sqrt 2}\varepsilon$,
there exists a unitary $V$ on $\caH$ such that
$\xi_{2}=V\xi_{1}$ and $\lV V-\unit_{\caH}\rV<\varepsilon$.
\end{lem}
\begin{notation}\label{exp}
For any $\varepsilon>0$, there exists a $\delta_{\nii}(\varepsilon)>0$ satisfying the following:
For any $t\in \bbR$ with $|t|\le\delta_{\nii}(\varepsilon)$,
we have $\lv e^{it}-1\rv<\varepsilon$.
We will fix such $\delta_{\nii}(\varepsilon)>0$ for each $\varepsilon>0$.
\end{notation}
\begin{thm}\label{lem6kos}
For any $\varepsilon>0$, there exists a $\delta_{\rk}(\varepsilon)>0$ satisfying the following:
For any Hilbert space $\caH$, unital $C^{*}$-algebra $\caA$ acting irreducibly on $\caH$ , and
$\xi,\eta\in \caH$,
 if there is a unitary operator $v$ on $\caH$ satisfying
 $\lV v-\unit \rV<\delta_{\rk}(\varepsilon)$
 and $\eta=v\xi$, 
 there exists a self-adjoint $h\in\caA$ such that
 $e^{ih}\xi=\eta$ and $\lV h\rV\le \delta_{\nii}(\varepsilon)$. 
\end{thm}
\begin{proof}
This is a quantitative version of the Kadison transitivity theorem.
It can be obtained by precise estimation of approximation in each step of 
the proof of the Kadison transitivity theorem.
\end{proof}
The following Theorem is called Glimm's Lemma 
\begin{thm}\label{glimm}
Let $\caA$ be a unital $C^{*}$-algebra acting on a Hilbert space $\caH$. Suppose that 
the intersection of $\caA$ and the set of compact operators on $\caH$ is $\{0\}$.
Then for any $\varepsilon>0$, a pure state $\varphi$ of $\caA$, a finite dimensional subspace $\caK$ of 
$\caH$ and a finite subset $\caF$ of $\caA$, there exists a unit vector
$\xi$ in the orthogonal complement of $\caK$
such that
\begin{align}
\lv
\varphi(x)-\braket{\xi}{ x\xi}
\rv<\varepsilon,\quad\text{for all}\quad x\in\caF.
\end{align}
\end{thm}

Let us recall the following fact from \cite{kos}.
\begin{lem}\label{lem28}
For any $\varepsilon>0$ and $n\in\nan$, there exists a
$\delta_{\nhh}(\varepsilon,n)>0$ satisfying the following.:
Let $\caH$ be an infinite dimensional Hilbert space and $\caA$
a unital $C^{*}$-algebra acting irreducibly on $\caH$.
Let $\{x_{i}\}_{i=1}^{n}\subset \caA$ be a finite sequence satisfying
$\sum_{i=1}^{n}x_{i} x_{i}^{*}=1$.
Let $\xi,\eta\in\caH$ be unit vectors such that
$x_{i}^{*}\xi$ and $x_{j}^{*}\eta$ are orthogonal for any $i,j=1,\ldots, d$ and
\begin{align}
\lv
\braket{\xi}{x_{i}x_{j}^{*}\xi}-\braket{\eta}{x_{i}x_{j}^{*}\eta}
\rv
<\delta_{\nhh}(\varepsilon, n),\quad i,j=1,\ldots, n.
\end{align}
Then there exists 
a positive element $h\in\caA_{+,1}$
such that
\begin{align}
\lV \bar h(\xi+\eta)\rV
<\frac 1{4\sqrt 2}\delta_{\rk}\lmk \frac \varepsilon 8\rmk e^{-\pi},\quad \text{and}\\
\lV \lmk \unit-\bar h\rmk (\xi-\eta)\rV
<\frac 1{4\sqrt 2}\delta_{\rk}\lmk \frac \varepsilon 8\rmk e^{-\pi}
\end{align}
hold for
\begin{align}
\bar h:=\sum_{j=1}^{n} x_{j}h x_{j}^{*}.
\end{align}
Here the function $\delta_{\rk}$ is given in Theorem \ref{lem6kos}.
\end{lem}
\begin{proof}
See section 3 of \cite{kos} for the proof.
\end{proof}

\end{document}